\documentclass[oneside]{amsart}
\usepackage{amsmath, amssymb}
\usepackage{amsthm}
\newtheorem{theorem}{Theorem}[section]
\newtheorem{proposition}[theorem]{Proposition}
\newtheorem{lemma}[theorem]{Lemma}
\newtheorem{corollary}[theorem]{Corollary}

\usepackage{graphicx}

\def\F{\mathcal{F} }

\def\R{\mathbb{R} } 
\def\Z{\mathbb{Z} } 
\def\S{\mathbb{S} } 
\def\T{\mathbb{T} } 
\def\nbd{neighborhood } 
\def\nbds{neighborhoods } 
\def\iff{if and only if }

\title[Preorder characterizations of lower separation axioms]
{Preorder characterizations of lower separation axioms and their applications to foliations and flows}
\author{Tomoo Yokoyama}
\date{\today}

\address{Department of Mathematics, Kyoto University of Education/JST PRESTO, 
1 Fujinomori, Fukakusa, Fushimi-ku Kyoto, 612-8522, Japan \\
}
\email{tomoo@kyokyo-u.ac.jp}

\thanks{The author is partially supported
by the JST PRESTO Program at Department of Mathematics, Kyoto University of Education
and by Grant for Basic Science Research Projects from The Sumitomo Foundation.}

\begin{document}

\maketitle

\begin{abstract} 
In this paper, we characterize several lower separation axioms $C_0, C_D$, $C_R$, $C_N$, $\lambda$-space, nested, $S_{YS}$, $S_{YY}$, $S_{YS}$,  and $S_{\delta}$ using pre-order.  
To analyze topological properties of (resp. dynamical systems) foliations,  we introduce notions of topology (resp. dynamical systems) for foliations. 
Then proper (resp. compact, minimal, recurrent) foliations  are characterized by separation axioms. 
Conversely, lower separation axioms are interpreted into the condition for foliations and several relations of them are described. 
Moreover, 
we introduce some notions for topologies from dynamical systems and foliation theory. 
\end{abstract}

\section{Introduction}

In this paper, 
we characterize $T_{-1}, T_{1/2}, T_{1/3},$ and $T_{1/4}$ topologies using order properties.  
These characterizations implies the invariance under arbitrary disjoint unions 
(i.e. the disjoint union of $T_{i}$ spaces is $T_{i}$ for $i = -1, 1/2, 1/3, 1/4$).  
Moreover, we show the following relations between topologies: 
$S_1$ $\Rightarrow$  $C_0$ $\Rightarrow$ recurrent.  
The following are equivalent for a topological space $X$:
(1) $X$ is $T_{YS}$; 
(2) $X$ is $T_0$ with $\mathop{\Downarrow}  x \cap \mathop{\Downarrow}  y  = \emptyset$ for any $x \neq y \in X$; 
(3) $X$ is $T_{1/4}$ and a downward forest. 
The following inclusion relations between topologies hold: $T_{YS}$ $\Rightarrow$ $T_{1/4}$.  
Moreover, we characterize several lower separation axioms 
$C_0, C_D$, $C_R$, $C_N$, $S_{YS}$, $S_{YY}$, $S_{Y}$, $\lambda$-space, nested, and  $S_{\delta}$ using pre-order.  
For instance, we have the following characterization for a topological space $X$:  
\\
$X : \lambda$-space $\Leftrightarrow$  $\overline{U} - U \subseteq \min X$ for any $\lambda$-closed subset $U$;  
\\
$X$: $C_R$ $\Leftrightarrow$ $X$: $S_{1}$ (i.e. $\min X = X$); 
\\
$X$: $C_N$ $\Leftrightarrow$ $\mathop{\downarrow} x$ is down-directed for any point $x$ of $X$; 
\\ 
$X$: $S_{Y}$ $\Leftrightarrow$ $\mathop{\mathrm{ht}}_{\tau} X \leq 1$ and the $T_0$-identification $\hat{X}$ of $X$ is $\min$-$\S^1$-free;   
\\ 
$X$: $S_{YS}$ $\Leftrightarrow$ $\mathop{\mathrm{ht}}_{\tau} X \leq 1$ and $\hat{X}$ is an downward forest.  

On the other hands, we introduce topological notions for foliations and flows. 
In fact, 
a codimension one foliation on a compact manifold is $S_{YS}$ if and only if it is $S_{1/2}$ such that 
$\overline{L_1} \cap \overline{L_2} = \emptyset$ for any leaves $L_1, L_2 \subset \mathrm{LD}$ with $\hat{L_1} \neq \hat{L_2}$, where $\overline{A}$ is the closure of a subset $A$ 
and $\hat{L} := \cup \{ L' \in \F \mid \overline{L} = \overline{L'} \}$. 
A foliation on a paracompact manifold is recurrent if and only if it is $C_0$. 
For codimension one foliations on compact manifolds, we show the following relations: 
pointwise almost periodic $\Rightarrow$ $S_{1/2}$ $\Rightarrow$ recurrent.  

In addition, 
we describe a sufficient condition that the set of $\F$-saturated open subsets of a decomposition on a topological space becomes a topology.

Conversely, we introduce some notions (recurrent, attracting, saddle, (weakly) hyperbolic-like, and exceptional) for topologies from dynamical systems and foliation theory, and characterize some notions using general topology theory. 
For instance, we show the following inclusions for a compact space with a flow $v$:  ``without weakly $\tau_v$-hyperbolic-like points'' $ \Rightarrow \tau_v$-recurrent  $\Rightarrow$ ``without $\tau_v$-hyperbolic-like points'', 
where $\tau_v$ is the quotient topology on the orbit space. 
Moreover if $X$ is a compact manifold, 
then each hyperbolic minimal set consists of weakly $\tau_v$-hyperbolic-like points.

\section{Preliminaries}

\subsection{Topological notions}

Let $X$ be a topological space. 
A subset $A$ of $X$ is saturated if 
$A$ is an intersection of open subsets, 
and is $\lambda$-closed if 
$A$ is an intersection of a saturated subset and a closed subset.  
Complements of $\lambda$-closed subsets are said to be $\lambda$-open. 
A topological space $X$ is a $\lambda$-space \cite{ADG} if 
the set of $\lambda$-open subsets becomes a topology. 
Clearly a topological space $X$ is a $\lambda$-space 
\iff the union of any two $\lambda$-closed sets is $\lambda$-closed. 
A subset $A$ of $X$ is kerneled if 
$A = \mathop{\mathrm{ker}} A$, 
where $\mathop{\mathrm{ker}} A := \cap \{U \in \tau \mid A \subseteq U \}$. 
For a subset $A$, denote by $\sigma A := \overline{A} - A$ the shell of $A$.  
A point $x$ is said to be closed (resp. open, kerneled) 
if so is $\{ x \}$. 
The derived set of a point $x \in X$ is the set of all limit points of the singleton $\{ x \}$. 
The kernel $\mathop{\mathrm{ker}} x$ of a point $x \in X$ is the intersection of all open \nbd of $x$. 
The shell of a point $x \in X$ is the difference $\mathop{\mathrm{ker}} x - \{ x \}$. 

A topological space is a Baire space if each countable intersection of dense open subsets of it is dense. 
A topological space is anti-compact if each compact subset is finite. 
In \cite{HM}, a topological space $X$ is nested if 
either $U \subseteq V$ or $V \subseteq U$ for any open subsets $U, V$ of $X$.  

Define the class $\hat{x}$ of a point $x$ of a topological space $(X, \tau)$ by 
$\hat{x} := \{ y \in X \mid \overline{x} = \overline{y} \}$, 
where $\overline{x}$ is the closure of a singleton $\{ x \}$. 
The quotient space of $X$ by the classes is denoted by $\hat{X}$ and called 
the class space of $X$. 
The quotient topology is denoted by $\hat{\tau}$. 
In other words, 
the class space $\hat{X}$ of $X$ is the quotient space $X/\sim$ defined by the following relation: 
$x \sim y$ if $\overline{x} = \overline{y}$. 
Note that the $\tau$-closure of a point $x$ is the disjoint union of the $\hat{\tau}$-closure of the class $\hat{x}$ (i.e. $\bigsqcup \overline{\hat{x}}^{\hat{\tau}} = \bigcup 
\{ \hat{y} \in \hat{X} \mid \hat{y} \in \overline{\hat{x}}^{\hat{\tau}} \}  = \overline{x}$ for any point $x \in X$, where $\bigsqcup$ is a disjoint union symbol).

\subsection{Notions of orders}

Recall that a pre-order (or quasi-order) is a binary relation on a set that is reflexive and transitive and a partial order is an antisymmetric pre-order. 
A set with a partial order is called a partially ordered set (also called a poset). 
A chain is a totally ordered subset of a poset. 
A subset is called a pre-chain if the $T_0$-identification of it is a chain. 
Two points $x$ and $y$ of a pre-ordered set is incomparable if neither $x \leq y$ nor $x \geq y$. 

Let $(X, \leq)$ be a pre-ordered set. 
Define the height $\mathop{\mathrm{ht}} x$ of a point $x \in X$ by 
$\mathop{\mathrm{ht}} x := \sup \{ |C| - 1 \mid C :\text{chain containing } x \text{ as the maximal point}\}$.  
The height $\mathop{\mathrm{ht}} X$ of $X$ is defined by 
$\mathop{\mathrm{ht}} X := \sup_{x \in X} \mathop{\mathrm{ht}} x$. 
Denote by $X_{i}$ (resp. $\min X, \max X$) the set of height $i$ points 
(resp. minimal points, maximal points). 
Moreover define   
the upset $\mathop{\uparrow}  x := \{ y \in X \mid x \leq y \} $ 
(resp. the downset $\mathop{\downarrow}  x := \{ y   \in X \mid y \leq x \} $) of 
a point $x \in X$, 
the class $\hat{x} := \{ y   \in X \mid y = x \} $) of $x$, 
and 
the derived set $\mathop{\Uparrow}  x :=  \mathop{\uparrow}  x - \{ x \}$ 
(resp. the shell $\mathop{\Downarrow}  x := \mathop{\downarrow}  x - \{ x \}$) of a point $x \in X$.  
For a pair $x \leq y \in X$, 
$[x,y] := \{ z \in X \mid x \leq z \leq y \}$,  
$[x,y) := \{ z \in X \mid x \leq z < y \}$,  
$(x,y] := \{ z \in X \mid x < z \leq y \}$,  
and $( x, y ) := \{ z \in X \mid x < z < y \}$. 
Then $[x, x] = \hat{x}$, $[x, y] = \mathop{\uparrow}  x \cap \mathop{\downarrow}  y$, and 
$(x, y) = \mathop{\Uparrow}  \hat{x} \cap \mathop{\Downarrow}  \hat{y}$. 
In general, $[x, y ] - \{ x, y \} = \mathop{\Uparrow}  x \cap \mathop{\Downarrow} y \neq (x, y)$. 
For an subset $A$ of $X$, 
$\mathop{\uparrow}  A := \bigcup_{x \in A} \mathop{\uparrow} x $ 
(resp. $\mathop{\downarrow}  A := \bigcup_{x \in A}  \mathop{\downarrow} x $), 
and 
$\mathop{\Uparrow}  \hat{A} := \mathop{\uparrow}  A - A$ 
(resp. $\mathop{\Downarrow}  \hat{A} := \mathop{\downarrow}  A - A$). 
Note that 
$\mathop{\Uparrow}  \hat{x} = \mathop{\uparrow}  x - \hat{x}$ and $\mathop{\Downarrow}  \hat{x} = \mathop{\downarrow}  x - \hat{x}$ for a pont $x \in X$. 
If $A = \mathop{\uparrow}  A$ 
(resp. $A = \mathop{\downarrow}  A$), 
then $A$ is  
called an upset (resp. a downset). 
Define the class poset $\hat{X} := X / \sim$ of a a pre-ordered set $X$ as follows: 
$x \sim y $ if $\hat{x} = \hat{y}$. 
Then $\hat{X}$ is a poset (i.e. antisymmetric). 
Denote by $\mathop{\hat{\uparrow}}, \mathop{\hat{\downarrow}}, 
\mathop{\hat{\Uparrow}}, \mathop{\hat{\Downarrow}}$ the operations 
with respect to the class poset. 
Then $\mathop{\hat{\uparrow}}  \hat{x} = \{ \hat{y} \in \hat{X} \mid x \leq y \}$,  
$\mathop{\hat{\downarrow}} \hat{x} = \{ \hat{y} \in \hat{X} \mid y \leq x \} $, 
$\mathop{\hat{\Uparrow}}  \hat{x} =  \mathop{\hat{\uparrow}}  \hat{x} - \{ \hat{x} \}$, and 
$\mathop{\hat{\Downarrow}}  \hat{x} = \mathop{\hat{\downarrow}}  \hat{x} - \{ \hat{x} \}$. 
Moreover, 
$\bigsqcup \mathop{\hat{\uparrow}} \hat{x} = \mathop{\uparrow}  x$,  
$\bigsqcup \mathop{\hat{\downarrow}} \hat{x} \mathop{\downarrow}  x$, 
$\bigsqcup \mathop{\hat{\Uparrow}} \hat{x} = \mathop{\Uparrow} \hat{x}$, and 
$\bigsqcup \mathop{\hat{\Downarrow}} \hat{x} = \mathop{\Downarrow} \hat{x}$

A downward (resp. upward) forest is a pre-ordered set of which the upset (resp. downset) of any point is a pre-chain. 
Notice that a poset is a downward forest if and only if 
$\mathop{\Downarrow}  x \cap \mathop{\Downarrow} y = \emptyset$ 
for any incomparable points $x, y$ of it. 
A downward (resp. upward) forest is a downward (resp. upward) tree if 
there is an point whose downset (resp. upset) is the whole set. 
Then the point is called the root of the downward (resp. upward) tree. 
Note that a downward (resp. upward) tree of height zero is a singleton 
and that a downward (resp. upward) tree of height one consists of 
height zero (resp. one ) points except the root (see Figure \ref{tree}). 
\begin{figure}
\begin{center}
\includegraphics[scale=0.4]{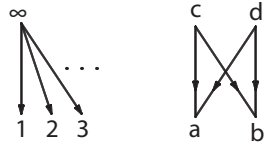}
\end{center}
\caption{A downward tree  $\{\infty \} \sqcup \{1, 2, 3,  \dots \}$ of height one and a $\min$-$\S^1$ $\{a, b, c, d\}$}
\label{tree}
\end{figure}

For downward forests $T_i$ with a base point $t_i \in \min T_i$, 
define a bouquet $\bigvee T_i$ of downward forests by $\bigvee T_i = \bigsqcup T_i /\sim$ and $x \sim y$ if $x = t_i$ and $y = t_j$ for some $i, j$. 
A partial ordered set $S = \{a, b, c, d\}$
is $\min$-$\S^1$ (see Figure \ref{tree}) if 
$\min S = \{ a, b \}$, $\max S = \{ c, d \}$, $a < c$, $a < d$, $b < c$, and $b < d$. 
We say that a pre-ordered set $P$ is $\min$-$\S^1$-free if it have no $\min$-$\S^1$ subsets.  

A subset $A$ of a pre-ordered set $P$ is down-directed 
if $\mathop{\downarrow} x \cap \mathop{\downarrow} y \neq \emptyset$ for any points $x, y \in A$. 
A point $x$ of a pre-ordered set $P$ is a top (resp. a bottom) 
if $y \leq x$ (resp. $y \geq x$) for any point $y \in P$. 
An upward forest (resp. pre-chain) $X$ is down-discrete if 
for any $x \in X - \min X$, there is a point $y < x$ such that 
$\mathop{\downarrow} x \cap \mathop{\uparrow} y  = \hat{x} \sqcup \hat{y}$. 
A subset $A$ of a pre-ordered set $P$ is (order-)convex if 
$\mathop{\uparrow} x \cap \mathop{\downarrow} y \subseteq A$ for any pair $x \leq y \in A$. 
Then a subset $A$ is convex if and only if $A = \mathop{\uparrow} A \cap \mathop{\downarrow} A$. 
Indeed, suppose $A$ is convex. 
Then $A \subseteq \mathop{\uparrow} A \cap \mathop{\downarrow} A$. 
For any point $x \in \mathop{\uparrow} A \cap \mathop{\downarrow} A$, 
there are points $a, b \in A$ such that $a \leq x \leq b$. The convexity implies $x \in A$. 
Conversely, suppose that $A = \mathop{\uparrow} A \cap \mathop{\downarrow} A$. 
For any pair $a \leq b \in A$ and any point $x$ with $a < x < b$, 
we have $x \in \mathop{\uparrow} A \cap \mathop{\downarrow} A = A$.

\subsection{Specialization pre-orders of topologies}

The specialization pre-order $\leq_{\tau}$ on a topological space $(X, \tau)$ is defined as follows: 
$ x \leq_{\tau} y $ if  $ x \in \overline{y}$. 
From now on, we equip a topological space $(X, \tau)$ with the specialization pre-order 
$\leq_{\tau}$.  
Then denote by $\mathop{\mathrm{ht}}_{\tau} x$ (resp. $\mathop{\mathrm{ht}}_{\tau} X$) 
the height of a point $x \in X$ (resp. a topological space $(X, \tau)$),  
and 
by $X_{i}$ the subset of height $i$ points. 
%
For a point $x$ of a topological space $(X, \tau)$, we have that 
$\mathop{\downarrow}  x = \overline{x}$,  
$\mathop{\uparrow}  x = \mathop{\mathrm{ker}} x$, 
and 
$\hat{x} = \mathop{\uparrow}  x \cap \mathop{\downarrow} x$, 
and that 
$\mathop{\Downarrow}  x$ is the derived set of $x$ 
and  
$\mathop{\Uparrow}  x$ is the shell of $x$.   
%
Notice that each closed subset is a downset with respect to the specialization pre-order 
and that the set $\min X$ of minimal points in $X$ is the set of points whose classes are closed. 
Moreover, 
each subset is saturated if and only if it is an upset. 
Indeed, 
since each saturated subset is the intersection of open subsets 
and open subsets are upsets, each saturated subset is an upset. 
Conversely, 
since the downset $\mathop{\downarrow} x$ of a point of of a topological space $X$ is closed, the complement $X - \mathop{\downarrow} x$ is open. 
Therefore the intersection $U = \bigcap_{x \in X - U} X - \mathop{\downarrow} x$ for an upset $U$ of $X$ is saturated. 

Define the saturation $\mathop{\uparrow} A$ of a subset $A$ of $X$ by defined by $\bigcup_{x \in A} \mathop{\uparrow} x$. 
Note that the saturation $\mathop{\mathrm{ker}} A$ of $A$ corresponds to the upset $\mathop{\uparrow} A$. 
Indeed, the definitions imply $\mathop{\uparrow} A \subseteq \mathop{\mathrm{ker}} A$. 
For a point $x \notin \mathop{\uparrow} A$, 
the difference $X - \mathop{\downarrow}x$ is an open \nbd of $A$. 
This implies $\mathop{\uparrow} A \supseteq \mathop{\mathrm{ker}} A$
 
For a subset $A$ of $X$, the intersection $\mathop{\uparrow} A \cap \overline{A}$ of the saturation of $A$ and the closure of $A$ is called the $\lambda$-closure of $A$. 
Then a subset $A$ is $\lambda$-closed if and only if $A$ corresponds to the $\lambda$-closure of $A$ (i.e. $A = \mathop{\uparrow} A \cap \overline{A}$). 
Indeed, suppose $A = \mathop{\uparrow} A \cap \overline{A}$. 
By the definition of $\lambda$-closed subsets, 
the intersection $A = \mathop{\uparrow} A \cap \overline{A}$ is $\lambda$-closed. 
Conversely, suppose that $A$ is $\lambda$-closed. 
Then $A \subseteq \mathop{\uparrow} A \cap \overline{A}$ 
and there is a saturated subset $B$ and a closed subset $C$ with $A = B \cap C$.  
Since a saturated subset containing $A$ contains a saturation of $A$ which is the upset $\mathop{\uparrow} A$, we obtain $\mathop{\uparrow} A \subseteq B$. 
Since the closure of $A$ is the minimal closed subset containing $A$, we have $\overline{A} \subseteq C$. 
Therefore $\mathop{\uparrow} A \cap \overline{A} \subseteq B \cap C = A$.  
Moreover, a $\lambda$-closed subset is order-convex. 
Indeed, for any points $a \leq b$ of a $\lambda$-closed subset $A$ and for any point $x \in X$ with $a \leq x \leq b$, since $\mathop{\downarrow} b = \overline{b} \subseteq \overline{A}$, we have $x \in\mathop{\uparrow} A \cap \overline{A} = A$. 

A topological space $(X, \tau)$ is Artinian if 
$(X, \leq_{\tau})$ satisfy the descending chain condition (i.e. each downward chain is finite).

\subsection{Separation axioms for points}

Let $(X, \tau)$ be a topological space. 
A point $x$ is $T_{-1}$ \cite{C}  (or $T_R$ \cite{W}) if 
it is either closed or 
there is a  \nbd $U$ of it with $U \nsupseteq \overline{x}$. 
A point $x \in X$ is $T_0$ if for any point $y \in X - \{ x \}$, 
there is an open subset $U$ of such that $\{x, y \} \cap U$ is a singleton.  
A point $x \in X$ is $T_{1/2}$ (resp. $T_{1/4}$)
if it is either closed or open (resp. closed or kerneled).    
A point $x \in X$  is $T_D$ \cite{AT} if 
the derived set $\mathop{\Downarrow} x$ is a closed subset
and  
it is $C_0$ \cite{W} if the derived set $\mathop{\Downarrow} x$ is not a union of nonempty closed subsets.   
Obviously, each $T_{1/2}$ topology is  $T_D$.  
For any $\sigma \neq -1$, a point $x$ in $X$ is $S_{\sigma}$ if the point $\hat{x}$ in $\hat{X}$ is $T_{\sigma}$. 
For instance, 
a point $x \in X$ is $S_{1/2}$ if and only if $\hat{x}$ is either an open point or a closed point in $\hat{X}$, 
it is $S_{1}$ if and only if $\hat{x}$ is a closed point in $\hat{X}$, and 
it is $S_{2}$ if and only if for any point $y, x \in X$ with $\hat{x} \neq \hat{y}$, 
there are disjoint open saturated \nbd $U_x, U_y$ of $x, y$.  
%
A point $x \in X$ is q-$S_2$ if for any point $y \in X$ with $\hat{x} \neq \hat{y}$, 
either there is a point $z \in X$ such that $x, y  \in \overline{z}$ or 
the pair $x, y$ can be separated by disjoint \nbds respectively. 

A point $x$ of $X$ is $T_{YS}$ \cite{AT} if 
$\mathop{\downarrow}  x \cap \mathop{\downarrow} y$ 
is either $\emptyset$, $\{ x \}$, or $\{ y \}$ for any $y \in X - \{ x \}$. 

In \cite{HM}, 
a point $x \in X$ is $S_{SD}$ if either $\hat{x}$ is closed or 
$\mathop{\Downarrow} \hat{x}$ is closed and is a class of some point, 
it is $S_\delta$ if either $\hat{x}$ is closed or $\mathop{\Downarrow} \hat{x}$ is a closure of some point. 
In \cite{AT}, 
a point $x \in X$ 
is $S_{Y}$ if $\mathop{\downarrow}  x \cap \mathop{\downarrow}  y$ contains at most one class for any $y \in X	$ with $\hat{x} \neq \hat{y}$. 
In \cite{W}, 
a point $x \in X$ is $C_D$ if the derived set $\mathop{\Downarrow} x$ of $x$ is either empty or non-closed, 
it is $C_R$ if the derived set $\mathop{\Downarrow}  x$ contains no nonempty closed subsets, 
and it is $C_N$ if there are no pair of two nonempty disjoint closed subsets in $\mathop{\Downarrow}  x$. 
These axioms satisfies the following relations \cite{W}: 
$T_1 
\Rightarrow 
C_R 
\Rightarrow  
C_0 
\Rightarrow  C_D$ 
and 
$C_R 
\Rightarrow  C_N$. 

\subsection{Separation axioms for topological spaces}

Recall that the separation axiom $S_{\sigma}$ is exactly 
the axiom $T_{\sigma}$ without $T_0$ axiom (i.e. $T_{\sigma} = S_{\sigma} + T_0$) 
for any $\sigma \neq -1$. 
Let $X$ be a topological space. 
A topological space is $T_{1/2}$ (resp. $T_{1/3}, T_{1/4}$) if 
each subset (resp. compact subset, finite subset) is $\lambda$-closed. 
Note that a topological space $X$ is $T_{1/2}$ (resp. $T_{1/4}$) \iff so is each point 
and that  $X$ is $T_{1/3}$ \iff for any compact subset $C$ of $X$ and for any point $x \in X - C$  there is a subset $A$ of $X$ with $C \subseteq A \subseteq X - \{x \}$ 
such that $A$ is closed or open  (see \cite{ADP}). 
A topological space is $T_{-1}$ (resp. $T_0$, $T_D$, q-$S_2$, 
etc) if so is each point. 
Note that $\tau$ is $T_0$ \iff $\leq_{\tau}$ is a partial order. 
Moreover, in \cite{HM}, a topological space $X$ is $S_{YY}$ if 
there is a point $p \in X$ such that $\mathop{\downarrow}  x \cap \mathop{\downarrow}  y$ 
is either $\emptyset$, $\hat{x}$, $\hat{y}$, or $\hat{p}$ for any $x,  y \in X$ with $\hat{x} \neq \hat{y}$. 
%
In \cite{M}, a topological space $X$ is said to be w-$R_0$ (resp. w-$C_0$) if $\bigcap_{x \in X} \overline{x} = \emptyset$ (resp. $\bigcap_{x \in X} \mathop{\mathrm{ker}}_{\tau} x = \emptyset$).   
It's known that $S_{1/2} \Rightarrow \lambda \text{-space } \Rightarrow S_{1/4}$ \cite{CUV}. 
The following relation holds (cf. \cite{ADP}): $S_{1/2} \Rightarrow S_{1/3} \Rightarrow S_{1/4}$.

\subsection{Decompositions}

Let $\F$ be a decomposition on $X$. 
For $x \in X$, denote by $L_x$ or $\F(x)$ the element containing $x$. 
A subset $A$ of $X$ is $\F$-saturated if $\F(A) = A$, where $\F(A) = \bigcup_{x \in A} L_x$. 
Then $\F$ is said to be  $S_{\sigma}$ (resp. recurrent, q-$S_2$, etc...) if so is the quotient space $X/{\F}$ of $\F$. 
Define the quotient $X/\hat{\F}$, called the class space of $\F$,  by 
$\hat {\F} := \{ \hat{L}  \mid L \in \F \}$, where $\hat {L} := \{ x \in X \mid \overline{L_x} = \overline{L} \}$. 
Note that the derived set $\mathop{\Downarrow_{\F}}{L} = \mathop{\downarrow_{\F}}{L} - L$ 
for an element $L \in \F$ does not correspond with the superior structure 
$\mathop{\Downarrow_{\hat{\F}}}{L}$ of $L \in \F$ in general, where  
$\mathop{\Downarrow_{\hat{\F}}}{L} := \cup \{ L' \in \F \mid L <_{\F} L' \} = \mathop{\downarrow_{\F}}{L} - L$.  
Note that a decomposition is pointwise almost periodic \iff it is $S_1$ 
and that a decomposition on a compact Hausdorff space is $R$-closed \iff it is $S_2$ \cite{Y3}. 
Moreover, we can obtain all results in this paper without the $T_0$ axiom, by replacing a point $x$ (resp. a separation axiom $T_{i}$) with a class $\hat{x}$ (resp. a separation axiom $S_{i}$).  

Denote by $\tau_{\F} := \{ \F(U) \mid U \in \tau \}$ the set of $\F$-saturation of open subsets. 
If $\tau_{\F}$ becomes a topology, then we call this the saturated topology on $X$. 
We also regard the saturated topology as a topology on the decomposition space $X/\F$, 
where the decomposition space $X/\F$ is defined by the quotient space of $X$ by the following equivalent relation: $x \sim y $ if $\F(x) = \F(y)$. 
In the next section, we show that $\tau_{\F}$ becomes a topology 
if $\F$ is either a foliation or a continuous action of a topological groups. 

In general, 
the set $\tau_{\F}$ of $\F$-saturated subsets is not a topology on the quotient space $X/\F$ even if $\overline{L}$ is $\F$-saturated for any $L \in \F$.  
For instance,  
for a decomposition $\F := \{ \{0 \} \times [0,1] \} \sqcup \{ (x, y) \mid x \neq 0, y \in [0,1] \}$ on $[0,1]^2$,  
the set $\tau_{\F}$ is not a topology.  

\subsection{Foliations}

Let $\F$ be a foliation on a paracompact manifold $M$. 
A subset is $\F$-saturated (resp. $\hat{\F}$-saturated) if it is a union of leaves (resp. leaf classes) of $\F$. 
Define a pre-order $\leq_{\F}$ for $\F$ by 
$L \leq_{\F} L'$ if $\overline{L} \subseteq \overline{L'}$. 
In other words, 
the pre-order $\leq_{\F}$ is the specialization pre-order of the saturated topology of $\F$. 
For a leaf $L \in \F$, define the leaf class $\hat{L} := \cup \{ L' \in \F \mid L =_{\F} L' \}$. 
Then the set of leaves (resp. leaf classes) is a decomposition. 
Denote by $L_x$ the leaf containing $x \in M$. 
The leaf class space $M/\hat{\F}$ is defined by the quotient space of $M$ by the following equivalent relation: 
$x \sim y $ if $\overline{L}_x = \overline{L}_y$. 
Then $M/\hat{\F} = \{ \hat{L} \mid L \in \F \}$. 
Note that the leaf class space $M/\hat{\F}$ is the $T_0$-tification of the leaf space $M/{\F}$.
We identify that $\tau_{\F}$ (resp. $\tau_{\hat{\F}}$) is the set of $\F$-saturated (resp. $\hat{\F}$-saturated) open subsets,  and that $\tau_{\F}$-closed (resp. $\tau_{\hat{\F}}$-closed) subsets are closed $\F$-saturated (resp. $\hat{\F}$-saturated) subsets. 
Denote by $\mathrm{Cl}$ (resp. $\mathrm{P}$) 
the unions of closed, proper (resp. non-closed) leaves. 
For a codimension one foliation, 
denote by $\mathrm{E}$ (resp. $\mathrm{LD}$) the unions of exceptional (resp. locally dense) leaves. 
A foliation is $T_{-1}$ (resp. $T_0$, $T_D$, q-$S_2$, recurrent, etc) if so is the saturated topology of it.

\subsection{Group-actions and flows}

By group-actions,  we mean continuous actions of topological groups on topological spaces. 
For a group-action $v$ and for any point $x$ of $X$, 
denote by $O_v(x)$ the orbit of $x$. 
Recall that a subset of $X$ is saturated with respect to $v$ if it is a union of orbits. 

By a flow, we mean either an $\R$-action on a topological space. 
For a flow $v : \R \times X \to X$, the $\omega$-limit (resp. $\alpha$-limit)  set of a point $x$ is $\omega(x) := \bigcap_{n\in \mathbb{R}}\overline{\{v_t(x) \mid t > n\}}$ 
(resp.  $\alpha(x) := \bigcap_{n\in \mathbb{R}}\overline{\{v_t(x) \mid t < n\}}$), 
where $v_t(x) := v(t, x)$.  
A point $x$ of $S$ is strongly recurrent (resp. recurrent) with respect to $v$ 
if $x \in \omega(x) \cap \alpha(x)$ (resp.  $x \in \omega(x) \cup \alpha(x)$).  
A flow is recurrent if each point is recurrent. 
%
Note that 
an orbit $O$ of $v$ is proper if and only if $\mathop{\Downarrow}_{\tau_v} O$ is closed,  
and that a point is recurrent if and only if it is $\tau_v$-closed or $\tau_v$-non-proper, 
where $\tau_v$ is the quotient topology of the orbit space $X/v$. 

A closed subset $F$ of $X$ is invariant if $v_t(F) = F$ for any $t \in \R$. 
A closed invariant set is minimal if it contains no proper closed invariant subsets. 
A minimal set $F$ is an attractor of $v$ if there is a \nbd $U  \supsetneq F$, called a basin of attraction,  with $F = \bigcap_{t > 0}v_t(U)$ such that $v_t(U) \subseteq U$ for any $t > 0$. 
A repellor is a reversed time attractor, a sink is an attractor which is a singleton, and a source is an repellor which is a singleton. 
A minimal set $F$ of $v$ is a saddle set \cite{B} if there exists a closed \nbd $U$ of $F$ such that 
$\overline{\{ x \in {U} - F \mid  O^{+}(x) \not\subseteq {U}, O^{-}(x) \not\subseteq {U}\}}\cap F \neq \phi$, where $O^+(x) := \{v_t(x) \mid t \geq 0 \}$ and $O^-(x) := \{v_t(x) \mid t \leq 0 \}$. 
A saddle set $F$ has countably many separatrices if the cardinality of orbits whose $\alpha$-limit set or $\omega$-limit set is $F$ is countable.

\subsection{Dynamical-system-like notions for points}

We say that a point is ($\tau$)-recurrent if it is $T_0$ or non-$T_D$ (i.e. either it is closed or the derived set is not closed), 
it is non-wandering if there is a subset which consists of ($\tau$)-recurrent points and whose closure is a \nbd of it,  
and it is proper if the derived set is closed. 
Obviously, each $S_1$ point is recurrent. 
A point $x$ of a topological space $X$ is ($\tau$-)exceptional if $x$ is neither maximal nor $T_D$, 
it is weakly ($\tau$-)non-indifferent (or attracting/repelling) 
if $\mathop{\uparrow} x$ is open, $\mathop{\Uparrow} \hat{x}$ consists of $T_D$ points,  and $\mathop{\Uparrow} x \neq \emptyset$,  
it is weakly ($\tau$-)saddle-like if either $\mathop{\uparrow} x$ is not open or there is an element $y > x$ such that 
$x \in \overline{(x,y] - \{ y \}}$, 
it is weakly ($\tau$-)hyperbolic-like if it is either weakly non-indifferent or weakly saddle-like, 
it is non-indifferent (resp. saddle-like) if it is weakly non-indifferent (resp. weakly saddle-like) and $\mathop{\Uparrow} \hat{x} \neq \emptyset$ contains $T_D$ points, 
it is hyperbolic-like if it is non-indifferent or saddle-like. 
Note that sinks and sources for a flow are non-indifferent but not saddle-like with respect to the orbit space. 
%
A topological space or a topology is of Anosov type if the set $\min X \neq X$ of minimal points is dense and there is a point whose closure is the whole space $X$. 
%
Note that the orbit (class) space of a hyperbolic toral automorphism is a topological space of Anosov type.


\section{Saturated topologies for decompositions}

Let  $\F$ be a decomposition on a topological space $(X, \tau)$. 
We observe a following statement. 

\begin{lemma}
The set $\tau_{\F}$ of saturations of open subsets is a topology
\iff  
$\tau_{\F}$ is invariant under taking finite intersections 
$(\mathrm{i.e.}$ the intersection of any two saturations of open subsets is a saturation of 
an open subset$)$. 
\end{lemma}

Now we state the sufficient condition that $\tau_{\F}$ becomes a topology.

\begin{lemma}\label{lem:001}
The set $\tau_{\F}$ of saturations of open subsets is contained in the topology $\tau$  
\iff  
the closure of each saturated subset is $\F$-saturated. 
In any case, the set $\tau_{\F}$ is a topology. 
\end{lemma}

\begin{proof} 
Suppose that $\tau_{\F}$ is a topology. 
Assume that there is an $\F$-saturated subset $A$ 
whose closure is not $\F$-saturated. 
Then there is a point $x \in \F(\overline{A}) - \overline{A}$. 
Put $U := X - \overline{A}$. 
This $U$ is an open \nbd of $x$. 
Since $\tau_{\F} \subseteq \tau$, 
the saturation $V := \F(U)$ is open with 
$V \cap A = \emptyset$. 
Then $V \cap \overline{A} = \emptyset$. 
Since $V$ is $\F$-saturated, 
we obtain $V \cap \F(\overline{A}) = \emptyset$. 
This contradicts that 
$V$ is a \nbd of $x \in \F(\overline{A})$. 
Thus $\overline{A}$ is $\F$-saturated for any $\F$-saturated subset $A \subseteq X$. 
Conversely, 
suppose that the closure of a saturated subset is $\F$-saturated. 
Fix any open subset $B$. 
Set $F := X - \F(B)$. 
Since $B \cap F = \emptyset$, we have $B \cap  \overline{F} = \emptyset$. 
The hypothesis implies that $\overline{F}$ is $\F$-saturated. 
Then $\F(B) \cap  \overline{F} = \emptyset$ and so $\F(B) \cap  \overline{X - \F(B)} = \emptyset$. 
This implies that $X - \F(B)$ is closed and so $\F(B)$ is open. 
Therefore $\tau_{\F} \subseteq \tau$. 
\end{proof}

In the cases of foliations or group-actions, we obtain the following statement.

\begin{proposition}\label{cor:02}
Suppose that 
$\F$ is either 
a foliation or the set of orbits of a group-action. 
Then the set $\tau_{\F}$ of saturations of open subsets is a topology on the quotient space $X/\F$.  
Moreover, 
the set $\tau_{\F}$ corresponds to the quotient topology. 
\end{proposition}

\begin{proof}
Since the closure of the saturation of any subset corresponds to the saturation of the closure of it, Lemma \ref{lem:001} implies the assertion. 
\end{proof}

%

\section{Characterizations of Lower separation axioms}

First,  we state observations of $T_{-1}, T_0, T_{1/4}, T_{1/3}$ and $T_{1/2}$ using the terms of orders. 

\begin{lemma}\label{lem:01-a}
A point $x$ of a topological space $X$ is $T_0$ \iff $|\hat{x}| = 1$. 
\end{lemma}

\begin{proof}
Suppose that $x$ is $T_0$. 
For any point $y \neq x$ of $X$, we have either $y \notin \overline{x}$ or $x \notin \overline{y}$. 
This means that $y \nleq x$ or $x \nleq y$. 
Thus $y \notin \hat{x}$. 
Conversely, suppose that $|\hat{x}| = 1$.  
Then $y \nleq x$ or $x \nleq y$. 
Hence $y \notin \overline{x}$ or $x \notin \overline{y}$. 
This shows that $x$ is $T_0$. 
\end{proof}


\begin{lemma}\label{lem:00} 
A topological space $X$ is $T_{-1}$ \iff the class of each minimal point of $X$ is closed (i.e. $\min X$ is $T_0$). 
\end{lemma}

\begin{proof}  
Suppose that $X$ is $T_{-1}$.  
Fix any minimal point $x$. 
Then $\overline{x} = \hat{x}$. 
Assume that $x$ is not closed. 
Then there is a \nbd $U$ of $x$ such that $U \nsupseteq \overline{x}$. 
Hence there is a point $y \in \overline{x} \setminus U$. 
This implies that $x \notin \overline{ y }$ and so $y \notin \hat{x} = \overline{x}$. 
This is a contradiction.  
Conversely, suppose that each minimal point of $X$ is closed. 
Fix any non-closed point $x$ of $X$. 
Then there is a point $y <  x$ and so 
$U :=X - \mathop{\downarrow} y $ is an open \nbd of $x$ with $U \nsupseteq \overline{x}$.  
\end{proof}

\begin{lemma}\label{lem:01} 
Let $(X, \tau)$ be a $T_0$ topological space. 
The following statements hold: 
\\
1. 
$\tau$ is $T_{1/4}$
\iff 
$\mathop{\mathrm{ht}}_{\tau} X \leq 1$.  
\\
2. 
$\tau$ is $T_{1/3}$
\iff 
for each compact subset $C$ of $X$ 
there are a closed subset $F$ and 
a downset $D$ such that 
$C = F \setminus D$. 
\\
3. 
$\tau$ is $T_{1/2}$
\iff 
$\mathop{\mathrm{ht}}_{\tau} X \leq 1$ and 
each height $1$ point is open.   
\end{lemma}

\begin{proof} 
Note that 
the set $\min X$ of minimal points with respect to the specialization order $\leq_{\tau}$ is 
the set of closed points. 

1. 
Suppose that 
$\tau$ is $T_{1/4}$.  
Fix any point $x \in X$. 
Then  
$x$ is closed or kerneled. 
If $x$ is closed, 
then 
$\mathop{\mathrm{ht}}_{\tau} x = 0$. 
It $x$ is not closed, then $\{x \} = \mathop{\mathrm{ker}}_{\tau} x = \mathop{\uparrow} x$. 
Therefore $x \in \min X \cup \max X$. 
This implies 
$X = \max X  \cup \min X$ 
and so 
$\mathop{\mathrm{ht}}_{\tau} X \leq 1$. 
%
Conversely, 
suppose that $\mathop{\mathrm{ht}}_{\tau} X \leq 1$. 
Fix any point $x \in X$. 
If $\mathop{\mathrm{ht}}_{\tau} x = 0$, 
then $x$ is closed. 
Otherwise $\mathop{\mathrm{ht}}_{\tau} x = 1$. 
For any point $y \neq x \in X$, 
we have 
$y \notin \mathop{\uparrow} x = \mathop{\mathrm{ker}}_{\tau} x$.  
Then $\mathop{\mathrm{ker}}_{\tau} x = \{ x \}$ 
and so $x$ is kerneled. 
Thus $\tau$ is $T_{1/4}$.  

2. 
Note that 
each saturated set is of form $X - D$ for some downset $D \subseteq X$. 
Suppose that 
$\tau$ is $T_{1/3}$. 
Fix a compact subset $C$ of $X$. 
Then $C$ is $\lambda$-closed. 
Since a $\lambda$-closed subset is 
the intersection of 
a saturated subset and 
a closed subset, 
there are a closed subset $F$ and 
a downset $D$ such that 
$C = F \setminus D$. 
Conversely, suppose that for each compact subset $C$ of $X$ there are a closed subset $F$ and a downset $D$ such that $C = F \setminus D$. 
Fix $C$, $F$, and $D$ as above. 
Since the complement $X - D$ is a saturated subset and $C = F \setminus D = F \cap (X - D)$, 
the compact subset $C$ is $\lambda$-closed. 

3. 
Recall that 
$\tau$ is $T_{1/2}$ 
\iff 
each point of it is either open or closed. 
Since each height $0$ point is closed, 
the first assertion 
implies the characterization of $T_{1/2}$. 
\end{proof}

By the definition of $S_{\sigma}$, 
we have the following statement. 

\begin{theorem}\label{cor} 
Let $(X, \tau)$ be a topological space. 
The following statements hold: 
\\
1. 
$\tau$ is $S_{1/4}$
\iff 
$\mathop{\mathrm{ht}}_{\tau} X \leq 1$.  
\\
2. 
$\tau$ is $S_{1/3}$
\iff 
for each compact subset $C$ of the class space $\hat{X}$ 
there are a closed subset $F$ of $\hat{X}$ and 
a downset $D$ of $\hat{X}$ such that 
$C = F \setminus D$. 
\\
3. 
$\tau$ is $S_{1/2}$
\iff 
$\mathop{\mathrm{ht}}_{\tau} X \leq 1$ and 
each height $1$ class is open.   
\end{theorem}

Note that 
the specialization order of the disjoint union $(X, \tau_X) := \bigsqcup_{\mu} X_{\mu}$ 
is the union of the specialization orders of $(X_{\mu}, \tau_{\mu})$ as 
subsets of  the product space $X \times X$ 
(i.e. $\leq_{\tau_X} = \bigsqcup_{\mu} \leq_{\tau_{\mu}}$), 
and that 
a compact subset $C$ of 
$X$ 
is contained in some finite disjoint union $\bigsqcup_{j =1}^k X_{\mu_j}$
such that 
each $C \cap X_{\mu_j}$ is compact in $X_{\mu_j}$. 
Combining with 
the above characterizations, 
we obtain 
the following invariance under arbitrary disjoint unions. 

\begin{proposition}\label{lem:037} 
Let $X_{\mu}$ be topological spaces and 
$i = -1, 1/2, 1/3$ or  $1/4$.  
Then 
the disjoint union $\bigsqcup_{\mu} X_{\mu}$ is  $T_i$ 
\iff 
$X_{\mu}$ is $T_i$ for any $\mu$. 
\end{proposition}

\begin{proof}  
Since 
$\min  (\bigsqcup_{\mu} X_{\mu}) =  \bigsqcup_{\mu} \min X_{\mu}$ 
and a disjoint union of $T_0$ topological spaces are $T_0$, 
Lemma \ref{lem:00} implies 
the assertion for $i = -1$. 
Since 
$\mathop{\mathrm{ht}}_{\tau} (\bigsqcup_{\mu} X_{\mu}) 
= \sup_{\mu} \mathop{\mathrm{ht}}_{\tau} X_{\mu}$,  
Lemma \ref{lem:01} implies 
the assertion for $i = 1/4$. 
Since 
the finite disjoint union 
of any downsets (resp. closed subsets)
is 
a downset (resp. closed subset), 
Lemma \ref{lem:01} implies 
the assertion for $i = 1/3$. 
Since each point in $X_{\mu_0}$ is 
open in $\bigsqcup_{\mu} X_{\mu}$ 
\iff so is it in $X_{\mu_0}$, 
Lemma \ref{lem:01} 
implies 
the assertion for $i = 1/2$. 
\end{proof}

%
%

Recall that a topological space is anti-compact if each compact subset is finite. 
Obviously, we obtain the following observation. 

\begin{lemma}\label{lem:036} 
Let $X_{\mu}$ be topological spaces. 
Then the disjoint union $\bigsqcup_{\mu} X_{\mu}$ is anti-compact
\iff $X_{\mu}$ is anti-compact for any $\mu$.  
\end{lemma}

Notice that 
each cofinite topological space with infinitely many elements is not anti-compact but $T_1$, because every subset is compact. 
Moreover there are $T_{1/3}$ topological spaces which are not $T_{1/2}$ (cf. Example 3.2. \cite{ADP}). 
Therefore we have the  negative answer for Question 3.4 \cite{ADP} as follows:  

\begin{proposition}\label{prop:01}
There are $T_{1/3}$ topological spaces which are neither $T_{1/2}$ nor anti-compact. 
\end{proposition}

\begin{proof}
Let $X$ be a topological space that is not $T_{1/2}$ but $T_{1/3}$, 
$Y$ a topological space which is not anti-compact but $T_{1/3}$. 
By Proposition \ref{lem:037}  and Lemma  \ref{lem:036}, 
the disjoint union $X \sqcup Y$ is neither anti-compact nor $T_{1/2}$ but $T_{1/3}$. 
\end{proof}

Using the characterization of $T_{1/4}$, 
we have the following characterization of $T_{YS}$. 

\begin{lemma}\label{lem61} 
The following are equivalent for a topological space $X$:
\\
1) 
$X$ is $T_{YS}$. 
\\
2)
$X$ is $T_0$ with 
$\mathop{\Downarrow}  x \cap \mathop{\Downarrow}  y  = \emptyset$ for any distinct pair $x \neq y \in X$. 
\\
3) 
$X$ is $T_{1/4}$ and a downward forest. 

In any case, each connected component of $X$ is the downset $\mathop{\downarrow} x$ for some point $x \in X$. 
\end{lemma}

\begin{proof}  
Note that  $X$ is $T_0$ in any case. 
Suppose that $X$ is $T_{1/4}$ and a downward forest. 
Then 
each connected component of $X$ is the downset $\mathop{\downarrow} x$ for 
some point $x \in X$. 
Fix any distinct pair $x \neq y \in X$. 
If the height of $x$ is zero, then 
$\Downarrow x = \emptyset$ 
and so $\mathop{\Downarrow}  x \cap \mathop{\Downarrow}  y  = \emptyset$. 
Thus we may assume that the height of $x$ (resp. $y$) is one. 
Then $x$ and $y$ are incomparable. 
Since $X$ is a downward forest, 
we have 
$\mathop{\Downarrow}  x \cap \mathop{\Downarrow} y = \emptyset$. 
This means that the condition 2 holds. 
Suppose that $X$ is $T_{YS}$. 
If there are points $x > y > z$ in $X$, 
then $y, z \in \mathop{\downarrow}  x \cap \mathop{\downarrow}  y$. 
This contradicts to $T_{YS}$ axiom. 
Thus $\mathop{\mathrm{ht}}_{\tau} X < 2$ 
and so $X$ is $T_{1/4}$. 
This implies that the condition 3 holds. 
Moreover we will show that the condition 2 holds.  
Fix any $x \neq y \in X$. 
If $\mathop{\downarrow}  x \cap \mathop{\downarrow}  y  = \{ y \}$, 
then the fact $\mathop{\mathrm{ht}}_{\tau} X \leq 1$ 
implies that $\mathop{\Downarrow}  y = \emptyset$. 
By the symmetry, 
we may assume that 
$\mathop{\downarrow}  x \cap \mathop{\downarrow}  y  = \emptyset$. 
Then 
$\mathop{\Downarrow}  x \cap \mathop{\Downarrow}  y  = \emptyset$. 
This means that the condition 2 holds. 
Conversely, 
suppose that 
$\mathop{\Downarrow}  x \cap \mathop{\Downarrow}  y  = \emptyset$ 
for any $x \neq y \in X$. 
If there are points $x > y > z$ in $X$, 
then $ z \in \mathop{\Downarrow}  x \cap \mathop{\Downarrow}  y$. 
This contradicts to the hypothesis. 
Thus $\mathop{\mathrm{ht}}_{\tau} X \leq 1$. 
Fix any $x \neq y \in X$. 
If  $x > y$, 
then the hypothesis implies that 
$\mathop{\downarrow}  x \cap \mathop{\downarrow}  y = \{y \}$. 
By symmetry, 
we may assume that 
$x$ and $y$ are incomparable. 
Then  the hypothesis implies that 
$\mathop{\downarrow}  x \cap \mathop{\downarrow}  y = \emptyset$. 
\end{proof}

\section{Characterizations of separation axioms}

We observe the following statement. 

\begin{lemma}
A point $x$ of a topological space $X$ is $S_D$ 
if and only if either $x$ is minimal or $x \notin \overline{\mathop{\Downarrow}  x}$. 
\end{lemma}

\begin{proof}
Suppose that $x$ is $S_D$. 
We may assume that $x$ is not minimal. 
Then $\mathop{\Downarrow}  x$ is closed and so 
$x \notin \mathop{\Downarrow}  x = \overline{\mathop{\Downarrow}  x}$. 
Conversely, suppose that either $x$ is minimal or $x \notin \overline{\mathop{\Downarrow}  x}$.
If $x$ is minimal, then $\hat{\mathop{\Downarrow}} \hat{x} = \emptyset$ is closed. 
Thus we may assume that $x \notin \overline{\mathop{\Downarrow}  x}$.
Since $\mathop{\downarrow}  x$ is closed and $\mathop{\downarrow}  x =  \{ x \} \sqcup  \mathop{\Downarrow}  x$,  the derive set $\mathop{\Downarrow}  x$ is closed. 
\end{proof}

Recall $\sigma U := \overline{U} - U$ for a subset $U$.  
Note that $\sigma U = \overline{U} \setminus \mathop{\uparrow} U$ for any $\lambda$-closed subset $U$. 
We have the following characterization of $\lambda$-spaces. 

\begin{proposition}\label{pro-1} 
A topological space $X$ is a $\lambda$-space 
\iff 
$\sigma U \subseteq \min X$ for any $\lambda$-closed subset $U$. 
\end{proposition}

\begin{proof}  
First we show that each condition implies that $\mathop{\mathrm{ht}}_{\tau} X \leq 1$. 
Indeed, 
suppose that $X$ is a $\lambda$-space. 
Assume that there are points $x > y > z$. 
Then  
$X - \mathop{\downarrow} y$ and $\mathop{\downarrow} z$ are  $\lambda$-closed
but $(X - \mathop{\downarrow} y) \cup \mathop{\downarrow} z$ is not $\lambda$-closed. 
This contradicts that the union of any two $\lambda$-closed sets is $\lambda$-closed. 
Suppose that $\sigma U  \subseteq \min X$ for any $\lambda$-closed subset $U$.  
The fact that the class of each singleton is $\lambda$-closed implies 
$\mathop{\mathrm{ht}}_{\tau} X \leq 1$. %

Suppose that 
there are a $\lambda$-closed subset $U$ and a point $x \in \sigma   U \setminus \min X$. 
Since $U$ is $\lambda$-closed, 
we have 
$U = \mathop{\uparrow} U \cap \overline{U}$ 
and so 
$\sigma U 
= \overline{U} - U 
= \overline{U} - (\mathop{\uparrow} U \cap \overline{U})  
= \overline{U} \setminus \mathop{\uparrow} U$. 
Therefore 
$x \in \sigma   U = \overline{U} \setminus \mathop{\uparrow} U$ 
and so 
$x \notin \mathop{\uparrow} U$. 
Fix a minimal point $y < x$. 
Then 
the class $\hat{y} \subseteq  \min X \setminus \mathop{\uparrow} U$ is $\lambda$-closed. 
We show that 
$U \sqcup \hat{y}$ is not $\lambda$-closed, where $\sqcup$ is a disjoint union symbol. 
Otherwise 
there is a closed subset $G$ of $X$ containing $U \sqcup \hat{y}$ 
such that 
$
U \sqcup \hat{y} = 
\mathop{\uparrow}(U \sqcup \hat{y}) \cap G = 
(\mathop{\uparrow}U \cap G) \cup (\mathop{\uparrow} y \cap G) $. 
Since $x \notin U \sqcup \hat{y}$ 
and $y < x$, 
we obtain  
$x \notin G$, 
which contradicts to 
$x \in \overline{U}$. 
This shows that 
$X$ is not a $\lambda$-space. 
Conversely, 
suppose that 
$X$ is not a $\lambda$-space. 
Then 
there are $\lambda$-closed subsets $A, B$ 
such that 
$A \cup B$ is not $\lambda$-closed. 
Then 
$A \cup B \subsetneq 
\mathop{\uparrow} (A \cup B) \cap \overline{A \cup B}$ 
and so 
there is a point 
$x \in (\mathop{\uparrow} (A \cup B) \cap (\overline{A} \cup \overline{B})) 
- A \cup B$.  
Since 
$A = \overline{A} \cap \mathop{\uparrow} A$ 
and 
$B = \overline{B} \cap \mathop{\uparrow} B$, 
we have 
$(\mathop{\uparrow} (A \cup B) \cap (\overline{A} \cup \overline{B})) 
- A \cup B 
= 
((\mathop{\uparrow} A \cap \overline{B}) 
\cup 
(\mathop{\uparrow} B \cap \overline{A}))\setminus A \cup B 
= 
((\mathop{\uparrow} A - A) \cap \sigma  B) 
\cup 
((\mathop{\uparrow} B - B) \cap \sigma  A) \subseteq X - \min X$.  
Thus 
$x \in \sigma  A \cup \sigma  B \setminus  \min X$. 
Hence 
$\sigma  A \nsubseteq \min X$  
or 
$\sigma  B \nsubseteq \min X$. 
\end{proof}

We have the following characterization of $S_{YS}$. 

\begin{proposition}\label{prop:a} 
The following are equivalent: 
\\
1. 
The topological space $X$ is $S_{YS}$. 
\\
2.  
$\mathop{\Downarrow} \hat{x} \cap \mathop{\Downarrow} \hat{y} = \emptyset$
for any $x, y \in X$ with $\hat{x} \neq \hat{y}$. 
\\
3. 
The class space $\hat{X}$ is a downward forest of height at most one. 
\end{proposition}

\begin{proof}  
In each case, we have $\mathop{\mathrm{ht}}_{\tau} X \leq 1$. 
Otherwise there are points $x > y > z$ in $X$ 
and so $z < y \in \mathop{\downarrow}  x \cap \mathop{\downarrow}  y$, 
which contradicts to $S_{YS}$ axiom 
(resp. the hypothesis in 2).  
The fact $\mathop{\mathrm{ht}}_{\tau} X \leq 1$ implies that 2 and 3 are equivalent. 
Suppose that $X$ is $S_{YS}$. 
Fix any $x, y \in X$ with $\hat{x} \neq \hat{y}$. 
If $y < x$, then the fact $\mathop{\mathrm{ht}}_{\tau} X \leq 1$ implies that $\mathop{\downarrow}  y = \hat{y}$ and so $\mathop{\Downarrow} \hat{x} \cap \mathop{\Downarrow} \hat{y} = \mathop{\Downarrow} \hat{y} = \emptyset$. 
By the symmetry, 
we may assume that 
$x$ and $y$ are incomparable. 
By  the definition of $S_{YS}$, 
we have  
$\mathop{\Downarrow}  \hat{x}  \cap \mathop{\Downarrow} \hat{y} = 
\mathop{\downarrow}  x  \cap \mathop{\downarrow} y =  \emptyset$. 
Conversely, 
suppose that 
$\mathop{\Downarrow} \hat{x} \cap \mathop{\Downarrow} \hat{y} = \emptyset$ 
for any $x, y \in X$ with $\hat{x} \neq \hat{y}$. 
Fix any $x, y \in X$ with $\hat{x} \neq \hat{y}$. 
If  $x > y$, 
then the hypothesis implies that 
$\mathop{\downarrow}  x \cap \mathop{\downarrow}  y = \hat{y}$. 
By symmetry, 
we may assume that 
$x$ and $y$ are incomparable. 
Then  the hypothesis implies that 
$\mathop{\downarrow}  x \cap \mathop{\downarrow}  y = 
\mathop{\Downarrow}  \hat{x}  \cap \mathop{\Downarrow} \hat{y} = \emptyset$. 
\end{proof}

\begin{proposition}\label{prop:b} 
The following statement holds: 
\\
1. 
$X$ is $S_{YY}$ \iff $\mathop{\mathrm{ht}}_{\tau} X \leq 1$ and $\hat{X}$ is a bouquet of downward forests.
\\
2. 
$X$ is $S_{Y}$ \iff $\mathop{\mathrm{ht}}_{\tau} X \leq 1$ and $\hat{X}$ is $\min$-$\S^1$-free.  
\end{proposition}

\begin{proof}
If there are points $x > y > z$ in $X$, 
then $\mathop{\downarrow}  x \cap \mathop{\downarrow} y$ 
contains two nonempty classes 
$\hat{y} \neq \hat{z}$. 
Thus 
we have $\mathop{\mathrm{ht}}_{\tau} X \leq 1$ in each case. 

1. 
Suppose that 
$X$ is $S_{YY}$. 
We may assume that 
$X$ is not $S_{YS}$. 
Then there is a point $p$ as in the definition of $S_{YY}$. 
Note that  
$X - \hat{p}$ is $S_{YS}$. 
This implies that 
$\hat{X} - \hat{p}$ is a downward forest. 
Therefore $\hat{X}$ is a bouquet of downward forests.
Conversely, 
suppose that 
$\hat{X}$ is a bouquet of downward forests of height at most one.
Then there is a minimal point $p$ 
such that 
$\hat{X} - \hat{p}$ is a downward forest of height at most one 
and so the complement $X - \hat{p}$ is $S_{YS}$. 
By the minimality of $p$, this implies that 
$X$ is $S_{YY}$. 

2. 
Suppose that 
$X$ is $S_{Y}$. 
For each pair of points $x, y \in X - \min X$ 
with $\hat{x} \neq \hat{y}$, 
the intersection $\mathop{\downarrow}  x \cap \mathop{\downarrow}  y$ 
contains at most one class. 
This implies that 
$\hat{X}$ is $\min$-$\S^1$-free.  
Conversely, 
suppose that 
$\hat{X}$ is $\min$-$\S^1$-free such that 
 $\mathop{\mathrm{ht}}_{\tau} X \leq 1$.   
Then 
the intersection $\mathop{\downarrow}  x \cap \mathop{\downarrow}  y$ 
for each pair of points $x, y \in X$ 
with $\hat{x} \neq \hat{y}$  
contains at most one class. 
Thus $X$ is $S_{Y}$. 
\end{proof}

We have the following characterization of $C_{0}$, $C_{D}$, $C_{R}$, and $C_{N}$ by using pre-order. 

\begin{proposition}\label{lem:01b} 
Let $x$ be a point of a topological space $X$. 
The following statement holds: 
\\
1. 
$x$ is $C_{0}$ \iff $x$ is minimal or $| \hat{x} | > 1$. 
\\
2. 
$x$ is $C_{D}$ \iff $x$ is minimal or $x \in \overline{\mathop{\Downarrow} x}$. 
\\
3. 
$x$ is $C_{R}$ \iff $\hat{x}$ is closed.
\\
4. 
$x$ is $C_{N}$ \iff $\overline{x}$ is down-directed. 
\end{proposition}

\begin{proof}  
1. 
Notice that unions of closed subsets are exactly downsets.  
Suppose that $x$ is $C_{0}$. 
Then the derived set $\mathop{\Downarrow} x$ is not a downset. 
Since $\mathop{\downarrow} x$ is a downset, 
there is a point $y \in \hat{x} - \{ x \}$ and so $| \hat{x} | > 1$. 
Conversely, 
suppose that $x$ is either minimal or not $T_0$. 
If $x$ is minimal, 
then $\hat{x} - \{ x \}$ is either empty or not a union of closed subsets. 
Thus we may assume that $x$ is not minimal. 
Since $| \hat{x} | > 1$, 
we obtain that $\mathop{\Downarrow} x$ is not a downset and so is not a union of closed subsets. 

2. 
Suppose that $x$ is $C_{D}$. 
We may assume that $x$ is not minimal. 
Then the derived set $\mathop{\Downarrow} x$ is either empty or non-closed. 
Since $x$ is not minimal, 
we have that $\mathop{\Downarrow} x = \mathop{\downarrow} x - \{ x\}$ is not closed and so $\mathop{\downarrow} x \supseteq \overline{\mathop{\Downarrow} x}$.
Since $ X - \mathop{\downarrow} x$ is open,  
we obtain 
$\mathop{\downarrow} x = \overline{\mathop{\Downarrow} x}$.  
Conversely, 
suppose that $x$ is minimal or $x \in \overline{\mathop{\Downarrow} x}$.  
If $x$ is minimal, then $\hat{x}$ is closed and so $\mathop{\Downarrow} x =  \hat{x} - \{ x \}$ is either empty or non-closed. 
If $x$ is not minimal, then the hypothesis implies that the derived set $\mathop{\Downarrow} x =  \mathop{\downarrow} x - \{ x \}$ is not closed. 

3. 
Suppose that $x$ is $C_{R}$. 
Assume that $\mathop{\mathrm{ht}}_{\tau} x \geq 1$. 
Then there is a point $y < x $ of $X$ 
and so $\mathop{\downarrow}  y \subset \mathop{\Downarrow} x$. 
This contradicts to the definition of $C_R$. 
Conversely, suppose that $\mathop{\mathrm{ht}}_{\tau} x = 0$.
Then $\mathop{\downarrow} x = \hat{x}$ and so $\mathop{\Downarrow} x =  \hat{x} - \{ x \}$. 
Since any point $y \in \hat{x}$ satisfies 
$\mathop{\downarrow} y  = \mathop{\downarrow} x$, 
the derived set $\mathop{\Downarrow} x$ contains no nonempty closed subsets.

4. 
Suppose that $x$ is not $C_{N}$. 
Then there are nonempty disjoint closed subsets $F, E \subseteq \mathop{\Downarrow} x$.
For any $y \in F$, $z \in E$, we have 
$\mathop{\downarrow} y \cap \mathop{\downarrow} z  = \emptyset$. 
This shows that $\mathop{\downarrow} x$ is not down-directed. 
Conversely, suppose that 
$\mathop{\downarrow} x$ is not down-directed. 
Then there are two points $y, z \in \mathop{\downarrow} x$ 
such that $\mathop{\downarrow} y \cap \mathop{\downarrow} z  = \emptyset$. 
This means that $x$ is not $C_N$. 
\end{proof}

We obtain the following inclusion relations:  $S_1 \Rightarrow C_0 \Rightarrow $ recurrent. 

\begin{lemma}\label{lem410}
Let $x$ be a point of a topological space $X$. 
\\
1) 
If $x$ is $S_1$, then $x$ is $C_0$. 
\\
2) 
If $x$ is $C_0$, then $x$ is recurrent. 
\end{lemma}

\begin{proof} 
The definitions of $C_0$ and recurrence 
imply the assertion 2). 
If $x$ is $S_1$,  
then $\hat{x}$ is closed and so 
$x$ is a minimal point. 
Lemma \ref{lem:01b} implies $x$ is  $C_0$. 
\end{proof}

There is a $T_0$-space which is recurrent but not $C_0$ (see Figure \ref{tree02}). 
Indeed, 
let $X$ be the set of natural numbers (i.e. $X := \Z_{\geq0}$). 
Define the topology $\tau$ as follows: 
a subset is closed if it is finite subset of $\Z_{>0}$ or the whole space. 
Note that the topology $\tau$ is the cofinite lower topology such that 
the set of elements of height zero (resp. one) is $X - \{ 0 \}$ (resp. \{ 0 \}).  
Then 
the height $1$ element $0$ is $T_0$. 
\begin{figure}
\begin{center}
\includegraphics[scale=0.4]{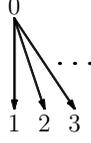}
\caption{A downward tree $\Z_{>0}$ which is recurrent but not $C_0$.}
\label{tree02}
\end{center}
\end{figure}

\section{Observations of separation axioms}

Now, we state the following observations. 

\begin{lemma}\label{lem:25}
The following statement holds: 
\\
1. 
$X$ is nested \iff ${X}$ is a pre-chain. 
\\
2.  
$X$ is $S_{SD}$ \iff $\hat{X}$ is an upward forest of height at most $1$. 
\\
3. 
$X$ is w-$R_0$ \iff $X$ has no bottoms with respect to $\leq_{\tau}$. 
\\
4. 
$X$ is w-$C_0$ \iff $X$ has no tops with respect to $\leq_{\tau}$. 
\\
5.  
$X$ is $S_{\delta}$ \iff $\hat{X}$ is a down-discrete upward forest.
\\
6. 
$\hat{X}$ is a downward forest \iff $\overline{x} \cap \overline{y} = \emptyset$ for any open subsets $U, V$ and any points $x \in U \setminus V$, $y \in V \setminus U$.  
\end{lemma}

\begin{proof}  
1. 
Obviously, if $\hat{X}$ is a chain, then $X$ is nested. 
Conversely, suppose $X$ is nested. 
For any points $x, y$ of $X$, 
since $X - \mathop{\uparrow} x$ and $X - \mathop{\uparrow} y$ are open, 
we have $\mathop{\downarrow} x \subseteq \mathop{\downarrow} y$ 
or $\mathop{\downarrow} y \subseteq \mathop{\downarrow} x$. 
Hence $x \leq y$ or $y \leq x$ for any points $x, y$ of $X$. 
This means that all points are comparable and so 
$\hat{X}$ is a chain. 
 
2. 
Note that an point whose class is closed is of height zero. 
Suppose that $X$ is $S_{SD}$. 
For any point $x \in X$ whose class is not closed, the derived set $\mathop{\Downarrow} \hat{x}$ of the class $\hat{x}$ is a class of some point 
and so the downset $\mathop{\downarrow} x$ is a pre-chain of height one. 
Therefore the class space $\hat{X}$ is an upward forest of height at most $1$. 
Conversely, 
suppose that the class space $\hat{X}$ is an upward forest of height at most $1$. 
For any point $x \in X$ whose class is not closed, 
the derived set $\mathop{\Downarrow} \hat{x}$ of the class $\hat{x}$ is a class of a minimal point 
and so is closed. 
 
3. 
By Proposition 1 \cite{M}, we have that 
$X$ is w-$R_0$ 
\iff 
$\mathop{\mathrm{ker}}_{\tau} x \neq X$ for any point $x$ of $X$. 
If there is a bottom $t$ of $X$, 
then $\mathop{\mathrm{ker}}_{\tau} t = \mathop{\uparrow} t = X$ 
and so $X$ is not w-$R_0$. 
If there are no bottoms of $X$, 
then  $\mathop{\mathrm{ker}}_{\tau} x = \mathop{\uparrow} x \neq X$ 
for any point $x \in X$, 
and so $X$ is w-$R_0$ .  

4. 
By Proposition 4 \cite{M}, we have that 
$X$ is w-$C_0$ 
\iff 
$\overline{x} \neq X$ for any point $x$ of $X$. 
The dual of the proof of 3 implies the assertion. 

5. 
Note that 
$X$ is $S_{\delta}$  
\iff 
$\mathop{\Downarrow} \hat{x}$ of a point $x$ is a down-discrete pre-chain. 
Since the latter condition is the definition of down-discrete upward forest, 
the assertion holds. 

6. 
Suppose that 
$\hat{X}$ is a downward forest. 
For any point $x \in X$, 
we have that 
$\mathop{\uparrow} x = \mathop{\mathrm{ker}}_{\tau} x$ is a pre-chain. 
Fix any open subsets $U, V$ and 
any $x \in U \setminus V$,  
$y \in V \setminus U$.  
Assume that  
there is a point $z \in \overline{x} \cap \overline{y}$, 
Then 
$\mathop{\uparrow} z$ is a pre-chain  
and so 
either 
$x \leq y$ or $y \leq x$. 
This means 
either 
$y \in U$ or $x \in V$, 
which contradicts to the hypothesis. 
Conversely, 
suppose that 
$\overline{x} \cap \overline{y} = \emptyset$ 
for any open subsets $U, V$ and 
any $x \in U \setminus V$,  
$y \in V \setminus U$.  
Assume that 
there is a point $z$ of $X$ such that 
$\mathop{\uparrow} z$ is not a pre-chain. 
Let $x, y \in \mathop{\uparrow} z$ be  the 
incomparable points. 
Then 
$U_x := X - \mathop{\downarrow} y$ 
and 
$U_y := X - \mathop{\downarrow} x$. 
Since 
$x \in U_y \setminus U_x$ 
and  
$y \in U_x \setminus U_y$, 
the hypothesis implies that 
$\mathop{\downarrow} x \cap \mathop{\downarrow} y = \emptyset$, 
which contradicts to $x, y \in \mathop{\uparrow} z$. 
\end{proof} 

Complementary, we call the topologies which satisfy the last condition in Lemma \ref{lem:25}
$S_{Q}$ topologies, 
because the last condition is similar to the dual of the $5$th condition 
(and Q looks like a refection image of $\delta$).  
Then we have the following inclusion relations for $S_Q$: 
\\
1) 
$S_{YS} = S _{1/4} \cap S_{Q}$. 
\\
2) 
$X$ is $S_{Q}$ and $S_{\delta}$ 
$\Rightarrow$ 
$\hat{X}$ is a disjoint unions of chains. 
\\
3) 
 $C_R$ or nested    
$\Rightarrow$ 
$S_Q$. 
%

\section{Dynamical-systems-like properties}

We state equivalence of recurrence between topological spaces and class spaces.  

\begin{lemma}\label{prop:06}
Let $(X, \tau)$ a topological space with the class space $(\hat{X}, \hat{\tau})$ 
and $p: X \to \hat{X}$ be the natural projection. 
Then $X$ is $\tau$-recurrent if and only if  
$\hat{x}$ is $\hat{\tau}$-recurrent for any $T_0$-point $x \in X$. 
\end{lemma}

\begin{proof}
Suppose that $\hat{x}$ is recurrent for any point $x \in X$ with $|\hat{x}| = 1$. 
Fix any point $x \in X$. 
If $|\hat{x}| > 1$, then $\hat{x} - \{ x \} \neq \emptyset$ and so the derived set $\overline{x} - \{ x \}$ is not closed. 
Thus we may assume that $|\hat{x}| = 1$. 
Then either $\hat{x}$ is a $\hat{\tau}$-closed point 
or the derived set $\overline{\hat{x}}^{\hat{\tau}} - \hat{x}$ is not $\hat{\tau}$-closed. 
If $\hat{x}$ is $\hat{\tau}$-closed, then 
$x$ is minimal and so $\overline{x} = \hat{x} = \{ x \}$.  
Thus we may assume that the derived set $\overline{\hat{x}}^{\hat{\tau}} - \hat{x}$ is not $\hat{\tau}$-closed. 
Since $\hat{x} = \{ x \}$, the derived set $\overline{x} - \{ x \} = p^{-1}(\overline{\hat{x}}^{\hat{\tau}} - \hat{x})$ is not closed. 
Conversely, suppose that $\tau$ is recurrent. 
Fix any point $\hat{x} \in \hat{X}$ with $\hat{x} = \{ x \}$. 
Then either $x$ is closed or $\overline{x} - \{ x \}$ is not closed. 
If $x$ is closed, then $\hat{x} = \{ x \}$ is also closed. 
Thus we may assume that $\overline{x} - \{ x \}$ is not closed. 
Then  the inverse image $p^{-1}(\overline{\hat{x}}^{\hat{\tau}} - \hat{x}) = p^{-1}(\overline{\hat{x}}^{\hat{\tau}}) - \{ x \} =  \overline{x} - \{ x \}$ 
is not closed and  so is $\overline{\hat{x}}^{\hat{\tau}} - \hat{x}$. 
\end{proof}

We state the relation of recurrence on a topological space and the class space.

\begin{lemma}\label{lem071}
Let $X$ be a topological space $X$ with the class space $\hat{X}$ and $p: X \to \hat{X}$ the natural projection.  
Then 
$p^{-1}(\{ \hat{x} \in \hat{X} \mid |\hat{x}| > 1\} \cup \hat{\mathrm{R}}) = \mathrm{R}$, 
where $\mathrm{R}$ is the set of recurrent points of $X$ and 
$\hat{\mathrm{R}}$ is the set of recurrent point of $\hat{X}$. 
\end{lemma}

\begin{proof}  
Recall that a recurrent point is either minimal or non-$T_D$ (i.e. the derived set is not closed). 
Fix a point $x \in \mathrm{R}$. 
If $x \in \min X$, then $p(x) \in \min \hat{X} \subseteq \hat{\mathrm{R}}$. 
If $x$ is not $T_0$, then $p(x) \in \{ \hat{x} \in \hat{X} \mid |\hat{x}| > 1\}$. 
Thus we may assume that $x$ is $T_0$ but not minimal. 
Then the derived set $\mathop{\Downarrow}  x$ of $x$ is not closed. 
Since $\mathop{\Downarrow}  x = p^{-1}(p(\mathop{\Downarrow}x))$ is not closed, 
so is the image $\hat{\mathop{\Downarrow}} \hat{x} = p(\mathop{\Downarrow} x)$. 
This means that $\hat{x} \in \hat{\mathrm{R}}$. 
On the other hands, fix a point $x \in p^{-1}(\{ \hat{x} \in \hat{X} \mid |\hat{x}| > 1\} \cup \hat{\mathrm{R}})$. 
If $|\hat{x}| > 1$, then $x$ is not $T_0$ and so is recurrent.  
Thus we may assume that $\hat{x} \in \hat{\mathrm{R}}$ and $x$ is $T_0$. 
If $\hat{x}$ is minimal, then so is $x$. 
Thus we may assume that $\hat{\mathop{\Downarrow}} \hat{x}$ is not closed. 
Since $x$ is $T_0$, 
the inverse image $\mathop{\Downarrow} x = p^{-1}(\hat{\mathop{\Downarrow}} \hat{x})$ is not closed. 
This implies that $x$ is recurrent. 
\end{proof}

The previous lemma implies an equivalence for the non-wandering property. 

\begin{proposition}\label{lem:07}
Let $(X, \tau)$ be a topological space $X$ with the class space $(\hat{X}, \hat{\tau})$.  
Then  $X$ is non-wandering if and only if 
the union of $\hat{\tau}$-recurrent points and 
points of $\hat{X}$ whose cardinality is more than one is dense in $\hat{X}$. 
\end{proposition}

We also state hyperbolic-like property for topological spaces. 

\begin{lemma}\label{lem:00a}
A weakly hyperbolic-like topological space has no open points. 
Conversely, a $T_D$ space without open points is weakly hyperbolic-like but not of Anosov type. 
\end{lemma}

\begin{proof} 
Suppose that $X$ is a weakly hyperbolic-like topological space. 
Note that each open point is maximal. 
Suppose there is an open point $x$ of $X$.
Then $\mathop{\Uparrow} x = \emptyset$ 
and $\mathop{\uparrow} x = \{ x \}$ is open. 
Therefore $x$ is not weakly hyperbolic-like. 
Conversely, suppose that $X$ is a $T_D$ space without open points.
Fix any point $x$ of $X$. 
If $\mathop{\uparrow} x$ is not open, then $x$ is weakly saddle-like. 
If $\mathop{\uparrow} x$ is open, then $\mathop{\Uparrow} x \neq \emptyset$ 
and so $x$ is weakly non-indifferent. 
Since the closure of a singleton is not the whole space, 
the space $X$ is not of Anosov type. 
\end{proof}

We state equivalences of conditions of saddle-like subsets. 

\begin{lemma}
The following conditions are equivalent for a point $x$ of a topological space $X$: 

1. $ x \in \overline{X - \mathop{\uparrow} x}$.  

2. $\mathop{\uparrow} x$ is not a \nbd of $x$ (i.e. $x \notin \mathrm{int} \mathop{\uparrow} x$). 

3. $\mathop{\uparrow} x$ is not open. 
\end{lemma}

\begin{proof}
If $\mathop{\uparrow} x$ is open, then $\mathop{\uparrow} x$ is a \nbd of $x$. 
If $\mathop{\uparrow} x$ is a \nbd of $x$, then the fact that each open subset is a upset 
implies that $\mathop{\uparrow} x$ is open. 
This means that the conditions $2$ and $3$ are equivalent. 
Since $\overline{X - \mathop{\uparrow} x} = X - \mathrm{int} \mathop{\uparrow} x$,  the conditions $1$ and $2$ are equivalent. 
\end{proof}

\begin{lemma}
The following conditions are equivalent for points $x < y$ of a topological space: 

1. $ x \in \overline{(x, y] - \{ y \}}$.  

2. $(x, y] - \{ y \} \neq \emptyset$.

3. $(x, y) \neq \emptyset$ or $|  \hat{y}| > 1$. 
\end{lemma}

\begin{proof}
Obviously, the conditions $2$ and $3$ are equivalent. 
If $ x \in \overline{(x, y] - \{ y \}}$, then $(x, y] - \{ y \} \neq \emptyset$.
Conversely, suppose that $(x, y] - \{ y \} \neq \emptyset$. 
Then there is a point $z \in (x, y] - \{ y \}$. 
Since $x < z$, we have $x \in  \mathop{\downarrow} z \subseteq \overline{(x, y] - \{ y \}}$. 
\end{proof}

\section{Applications for flows}


Let $v$ be a flow on a compact space $X$ 
and $\tau_v$ the quotient topology of the orbit space $X/v$. 
Recall that the orbit space consists of orbits (i.e. each class is an orbit). 
Note that $v$ is pointwise periodic \iff $\tau_v$ is $T_1$. 
Recall that a subset of $X$ is saturated with respect to $v$ if it is a union of orbits. 
Lemma \ref{lem:00} implies the characterization of $T_{-1}$-separation property of the flow $v$. 

\begin{lemma} 
The quotient topology $\tau_v$ of a homeomorphism $v$ 
on a compact metrizable space is $T_{-1}$ if and only if each minimal set is a closed orbit.  
\end{lemma}


Now we apply the above results to flows.

\begin{lemma}
The following statement hold: 
\\
1.  
$\tau_v$ is w-$C_0$ 
\iff 
$v$ is not transitive. 
\\
2.  
$\tau_v$ is $S_{Q}$ \iff $x \in \overline{O_v(y)}$ or $y \in \overline{O_v(x)}$ for any points $x, y$ with $\overline{O_v(x)} \cap \overline{O_v(y)} \neq \emptyset$. 
\\
3. 
$\tau_v$ is $C_N$ \iff each orbit closure contains exactly one minimal set. 
\end{lemma}

\begin{proof}
1. Since a dense orbit corresponds to a top with respect to the quotient topology, 
the assertion holds. 

2. 
Suppose that $\tau_v$ is $S_{Q}$. 
Then the quotient space $X/v$ is a downward forest.  
For any points $x, y$ with 
$\overline{O_v(x)} \cap \overline{O_v(y)} \neq \emptyset$, 
fix $z \in \overline{O_v(x)} \cap \overline{O_v(y)}$. 
Since the upset of a point is a pre-chain, 
we have either $z \leq_{\tau} x \leq_{\tau} y$ or $z \leq_{\tau} y \leq_{\tau} x$. 
This means $x \in  \overline{O_v(y)}$ or $y \in \overline{O_v(x)}$. 
Conversely, 
suppose $x \in  \overline{O_v(y)}$ or $y \in \overline{O_v(x)}$ 
for any points $x, y$ with 
$\overline{O_v(x)} \cap \overline{O_v(y)} \neq \emptyset$. 
Assume that the quotient space $X/v$ is not a downward forest.  
Then there are a point $z \in X$ and a pair $x, y \in \mathop{\uparrow}_{\tau} z$ of incomparable elements. 
Since $z \in \overline{O_v(x)} \cap \overline{O_v(y)}$, 
the hypothesis implies that $x \in \overline{O_v(y)}$ or $y \in \overline{O_v(x)}$. 
Therefore either $x \leq y$ or $y \leq x$, 
which contradicts to the incomparability of $x$ and $y$. 
Thus the quotient space $X/v$ is a downward forest.

3. 
Since $X$ is compact, each orbit closure contains at least one minimal set. 
Suppose that $\tau_v$ is $C_N$. 
By the down-directed property, each orbit closure contains contains exactly one minimal set. 
Conversely, suppose that each orbit closure contains contains exactly one minimal set. 
This implies that the closure of each point is  down-directed. 
By Proposition \ref{lem:01b}, $\tau_v$ is $C_N$. 
\end{proof}

We characterize the recurrence of a flow.

\begin{theorem} 
Then the quotient topology $\tau_v$ of a flow $v$ on a compact metrizable space is $C_{D}$ if and only if $v$ is recurrent. 
\end{theorem}

\begin{proof}
Note that the derived set $\mathop{\Downarrow}_{\tau_v} O$ for a proper orbit $O$ is closed. 
Suppose that the quotient topology $\tau_v$ is $C_{D}$. 
Then each proper orbit is closed. 
Conversely, suppose that $v$ is recurrent. 
Then each proper orbit is closed. 
Fix a recurrent orbit $O$ which is not closed. 
By Corollary 2.3 \cite{Y2}, 
the orbit class of $O$ contains uncountably many points 
and so $O \in \overline{\hat{O} - O}$. 
\end{proof}

\subsection{Hyperbolic-like property for topology}

Let $v$ be a flow on $X$ and $\tau_v$ the quotient topology of the orbit space $X/v$.

\begin{lemma}\label{lem:23}
A $\tau_v$-recurrent orbit is $v$-recurrent. 
If $X$ is locally compact Hausdorff, then the converse holds. 
\end{lemma}

\begin{proof}
Suppose that $O_x$ is an $\tau_v$-recurrent orbit. 
Then either $O_x$ is closed or $\overline{O_x} - O_x$ is not closed. 
If $O_x$ is closed, then $x \in \alpha(x) \cup \omega(x)$ trivially. 
Thus we may assume that $\overline{O_x} - O_x$ is not closed. 
This implies that $x \in \overline{\overline{O_x} - O_x}$ and so $x  \in \alpha(x) \cup \omega(x)$. 
Second, suppose that $X$ is locally compact Hausdorff and $x$ is a recurrent point. 
We may assume that $O_x$ is not closed. 
We show that $\overline{O_x} - O_x$ is not closed. 
Otherwise $\overline{O_x} - O_x$ is closed. 
Applying the Baire category theorem, 
since $\overline{O_x}$ is locally compact Hausdorff and so Baire and 
since each open subset of a Baire space is a Baire space, 
the orbit $O_x$ is Baire. 
Let $U_n := v(\R - [n, n+1], x) \subset O_x$ for $n \in \Z$. 
Since $O_x$ is $v$-recurrent, each $U_n$ is open dense in $O_x$. 
Since $O_x$ is Baire, we have $\bigcap_n U_n$ is dense, which contradicts to the definition of $U_n$. 
Therefore $O_x$ is $\tau_v$-recurrent.  
\end{proof}

There is a transitive flow  on a metrizable space such that each regular point is not $\tau_v$-recurrent but $v$-recurrent. 
Indeed,  
applying a dump function to an irrational rotation, consider a vector field on $\T^2$ with one singular point $x$ such that each regular orbit is dense. 
Let $X  := O \sqcup \{ x \}$ be the union of $x$ and a regular orbit $O$ and $v$ the restriction of the flow. 
Then $X$ is metrizable and $v$ consists of one $v$-recurrent non-closed orbit $O$ and one singularity $x$. 
Thus $\overline{O} -O = \{ x \}$ is closed and so $O$ is not $\tau_v$-recurrent but $v$-recurrent.  

We state connectivity of $\alpha$-limit sets.

\begin{lemma}\label{01}
Any $\alpha$-limit set of a flow on a sequentially compact space $X$ is connected. 
\end{lemma}

\begin{proof}
Since an orbit is a continuous map of $\R$ and so is connected, 
it is contained in a connected component of $X$. 
Thus we may assume that $X$ is connected .
Assume that there is a disconnected $\alpha$-limit set $\alpha (z)$. 
Put $U$ and $V$ disjoint open subsets each of which intersects $\alpha (z)$ with $\alpha (z) \subseteq U \sqcup V$. 
Since $X$ is connected, the complement $A := X - (U \sqcup V)$ is nonempty and closed. 
The sequential compactness implies that $A$ is sequentially compact. 
For any $T > 0$, there are numbers $t_U, t_A, t_V > T$ such that 
$v^{t_U}(z) \in U, v^{t_A}(z) \in A, v^{t_V}(z) \in V$. 
Since $A$ is sequentially compact, we have $A \cap \alpha(z) \neq \emptyset$, 
which contradicts to $\alpha (z) \subseteq U \sqcup V$. 
\end{proof}

We state a property of attractors.

\begin{lemma}\label{lem02}
Let $v$ be a flow on a sequentially compact space $X$ and $F$ an attractor with a basin $U$ of attraction.  
Any point $z \in U$ with $\alpha (z) \cap F \neq \emptyset$ is contained in $F$. 
\end{lemma}

\begin{proof}
Note that each orbit of a flow is contained in a connected component. 
Let $z \in U$ be a point with $\alpha (z) \cap F \neq \emptyset$. 
Since $v^t(z) \in U$ for any $t > 0$, we have $\alpha (z) \subseteq U$. 
Then $ \alpha (z) \subseteq \bigcap_{t>0} v^t(U) = F$. 
Then there is a large number $N>0$ such that $v^{-t-N}(z) \in U$ for any $t \geq 0$. 
Then $v^{-N} (z) = v^{t}(v^{-t-N} (z)) \in v^{t}(U)$ and so $v^{-N}(z) \in \bigcap_{t > 0}v^t(U) = F$. 
This implies $z \in F$. 
\end{proof}

We show that $v$-attractors are non-indifferent.  

\begin{proposition}\label{lem:24}
An orbit contained in an attractor $F \subsetneq X$ for a flow on a sequentially compact  space $X$ is non-indifferent with respect to $\tau_v$. 
\end{proposition}

\begin{proof} 
Let $x$ be a point of an attractor $F$ for $v$ and $U$ a basin of attraction of $F$. 
Then $\mathop{\downarrow}_{\tau_v} O_x = F$ and $U - F \neq \emptyset$. 
The minimality implies that  $\hat{x} = F$ is closed. 
Let $V = \bigcup_{t \in \R}v^t(U)$. 
We show that $\mathop{\uparrow}_{\tau_v} O_x = V$. 
Indeed, fix any point $y \in V$. 
By the hypothesis, the fact $F = \bigcap_{t > 0}v^t(U)$ implies $\overline{O_y} \cap F \neq \emptyset$. 
The minimality of $F$ implies $F \subset \overline{O_y}$. 
Thus $V \subseteq \mathop{\uparrow}_{\tau_v} O_x$. 
Fix a point $y \notin V$. 
Since $V$ is invariant open, we obtain $\overline{O_y} \cap F = \emptyset$. 
Therefore $\mathop{\uparrow}_{\tau_v} O_x = V$. 
Since $v^t$ is a homeomorphism, the subset $V = \mathop{\uparrow}_{\tau_v} O_x$ is open.  
Fix any point $z \in V - F$. 
Then $\omega (z) = F$. 
We show that $\alpha (z) \cap V = \emptyset$. 
Otherwise $\alpha (z) \cap V \neq \emptyset$. 
Then the intersection $\alpha (z) \cap V$ is saturated and closed in $V$. 
Put $w \in \alpha (z) \cap V$. 
Since $F \cap \overline{O(y)} \neq \emptyset$ for any point $y \in V$, 
we have that $\emptyset \neq F \cap \overline{O(w)} \subseteq F \cap \alpha (z)$ and so $z \in F$, 
which contradicts to the choice of $z$. 
Thus the orbit of any point $z \in V - F$ is closed in the open subset $V- F$. 
Therefore $\mathop{\Downarrow}_{\tau_v} O_z = \overline{O_z} - O_z = \overline{O_z} \setminus (V - F) = \alpha (z) \cup \omega (z)$ is closed and so $O_z$ is $T_D$.  
Therefore $\mathop{\Uparrow}_{\tau_v} O_x - O_x \neq \emptyset$ consists of $T_D$ orbits. 
\end{proof}

We state the relation between saddle sets and weakly $\tau_v$-saddle-like sets. 

\begin{lemma}\label{25} 
Let $F$ be a saddle set with at most countably many separatrices of a flow $v$ on a sequentially compact space $X$. 
If each \nbd of $F$ intersects uncountably many orbits, then each orbit in $F$ is weakly $\tau_v$-saddle-like. 
\end{lemma}

\begin{proof} 
Let $v$ be a flow on a sequentially compact space $X$ and $F$ a saddle set with countably many separatrices. 
Since $F$ is a saddle set, we obtain $\mathop{\Uparrow}_{\tau_v} O_x \neq \emptyset$ and $F = \overline{O_x}$. 
Suppose that $x$ is not weakly $\tau_v$-saddle-like. 
Then $\mathop{\uparrow}_{\tau_v} O_x$ is an open \nbd of $F$ and 
$x \notin  \overline{(\mathop{\Uparrow}_{\tau_v} O_x \cap \mathop{\downarrow}_{\tau_v} O_y) - O_y}$ for any point $y > x$. 
%
Fix a point $y \in V -F$ with $x < y$. 
Then $\emptyset  = \mathop{\uparrow}_{\tau_v} O_x \cap ((\mathop{\Uparrow}_{\tau_v} O_x \cap \mathop{\downarrow}_{\tau_v} O_y) - O_y) 
= (\mathop{\Uparrow}_{\tau_v} O_x \cap \mathop{\downarrow}_{\tau_v} O_y) - O_y$ and so $\mathop{\Uparrow}_{\tau_v} O_x \cap \overline{O_y} =  O_y$. 
Since $\mathop{\Uparrow}_{\tau_v} O_x$ is an open \nbd of $O_y$, 
the orbit $O_y$ is proper and so $F \subseteq \overline{O_y} - O_y = \alpha(y) \cup \omega(y) \subseteq X - \mathop{\Uparrow}_{\tau_v} O_x 
= F \sqcup (X - \mathop{\uparrow}_{\tau_v} O_x)$. 
Therefore either $\alpha (y) \subseteq F$ or $\omega (y) \subseteq F$. 
Since $\mathop{\uparrow}_{\tau_v} O_x$ intersects uncountably many orbits, 
we have uncountably many separatrices of $F$, 
which contradicts to the countability of separatrices.
\end{proof}

Summarize the relations between topological properties and dynamical properties. 

\begin{theorem}\label{prop:02}
Let $x$ be a point of a sequentially compact space $X$. 
The following holds: 
\\
1) $O_x$ is $\tau_v$-recurrent  $\Leftrightarrow$ $O_x$ is $v$-recurrent. 
\\
2) ${O}_x $ is non-indifferent $\Leftarrow$ $\overline{O_x}$ is an attractor.  
\\ 
3) $O_x$ is weakly $\tau_v$-saddle-like $\Leftarrow$ $\overline{O_x}$ is a saddle set with countably many separatrices and each \nbd of it intersects uncountably many orbits. 
\end{theorem}

Now we state a relation between  $\tau_v$-recurrence and hyperbolic-like property. 

\begin{proposition}\label{lem:26} 
Let $v$ be a flow on a topological space $X$. 
If $X$ is $\tau_v$-recurrent, then $X$ has no hyperbolic-like points. 
Conversely, 
if $X$ is compact and has no weakly hyperbolic-like minimal points, then $X$ is $\tau_v$-recurrent. 
\end{proposition}

\begin{proof}
Suppose that there is a hyperbolic-like point $x$ of the space $X$. 
Then there is an orbit $O_y >_{\tau_v} O_x$ which is $\tau_v$-$T_D$. 
Obviously $O_y$ is not $\tau_v$-recurrent. 
Conversely, suppose that $X$ is compact but not $\tau_v$-recurrent.  
Then there is a non-closed proper orbit $O$. 
Since $X$ is compact, there is an orbit $O' \subsetneq \overline{O}$ contained in a minimal set. 
It suffices to show that there is a weakly hyperbolic-like minimal point. 
Indeed, we may assume that there are no weakly non-indifferent orbits which are minimal 
and each upset of minimal point is $\tau_v$-open. 
Then there is a non-$T_D$-orbit $O'' >_{\tau_v} O'$. 
Since $\mathop{\Uparrow}_{\tau_v} O'$ is an open \nbd of $O''$, 
the non-$T_D$ property implies $\overline{O''} \cap \mathop{\Uparrow}_{\tau_v} O' \neq O''$ and so $(\mathop{\Uparrow}_{\tau_v} O_x \cap \overline{O_y}) -  O_y \neq \emptyset$. 
This implies that $O'$ is a weakly $\tau_v$-saddle-like.  
\end{proof}

The converses of the above statements do not hold. 
In fact, the trivial topology on a two point set is compact recurrent and there are weakly non-indifferent points. 
The one point compactification $X$ of an infinite discrete space $Y$ with the lower cofinite topology (i.e. the closed sets of except $X$ correspond to finite subsets of $Y$) is also $T_0$ recurrent and there are weakly saddle-like points. 
Moreover, a four point space $X = \{a, b, c, d\}$ with a topology $\{ \emptyset, \{ c \}, \{a, b\}, \{a, b, c\}, X\}$ is not recurrent but compact and has no hyperbolic-like points. 

\subsection{Weak hyperbolic-like property for vector fields}

On the hyperbolic minimal set of a vector field, we have the following statement. 

\begin{proposition}\label{prop:d} 
Each hyperbolic minimal set of a $C^1$ vector field $v$ on a compact manifold consists of weakly $\tau_v$-hyperbolic-like points. 
\end{proposition}

\begin{proof} 
Let $\mathcal{M}$ be a hyperbolic minimal set. 
Suppose that $\mathcal{M}$ is attracting or repelling.  
Then the weakly either stable or unstable manifold of it is an open \nbd of it. 
This means each orbit contained in $\mathcal{M}$ is non-indifferent. 
Suppose that $\mathcal{M}$ is neither attracting nor repelling.  
Then a local weakly (un)stable manifold of $\mathcal{M}$ is nowhere dense 
and so the weakly (un)stable manifold $W^{\sigma}(\mathcal{M})$ of  $\mathcal{M}$ 
has empty interior ($\sigma =u, s$). 
Then the $\tau_v$-kernel of any orbit contained in $\mathcal{M}$ is 
contained in $W^s(\mathcal{M}) \cup W^u(\mathcal{M})$ and so is not open. 
Then each orbit contained in $\mathcal{M}$ is weakly  $\tau_v$-saddle-like. 
\end{proof}

\section{Applications for foliations}

Lemma \ref{lem:00} implies the following observation. 

\begin{lemma}
A foliation on a compact manifold is $T_{-1}$ \iff each minimal set is a compact leaf. 
\end{lemma}

The recurrence for foliations implies the $C_0$ separation axiom as follows.

\begin{theorem} 
The following are equivalent for 
a foliation $\F$ on a compact manifold: 
\\
1) 
$\F$  is recurrent.  
\\
2) 
$\mathrm{P} = \emptyset$.  
\\
3) 
$\F$ is $C_0$.  
\end{theorem}

\begin{proof}
Lemma \ref{lem410} implies that 
a $C_0$ foliation on a compact manifold is recurrent. 
It suffices to show that 
recurrence implies the $C_0$ separation axiom. 
Suppose that $\F$  is recurrent. 
Then there are no proper non-compact leaves. 
Theorem 2.1 \cite{Y2} implies that 
each non-proper leaf 
is not $T_0$. 
Therefore 
each leaf is compact or non-$T_0$ 
and so minimal or non-$T_0$. 
Thus $\F$ is $C_0$.  
\end{proof}

By the same argument, Corollary 2.3 \cite{Y2} implies a similar statement for group-actions. 

\begin{corollary}\label{cor}
A group-action on a paracompact manifold is recurrent if and only if it is $C_0$. 
\end{corollary}

The $S_{1/2}$ separation axiom implies the recurrence. 

\begin{lemma}\label{01a}
A $S_{1/2}$ decomposition $\F$ on a compact space $X$ is recurrent.  
\end{lemma}

\begin{proof}
Fix an element $L \in \F$. 
Then either $\overline{L}$ is a minimal set or $\hat{L}$ is open.  
If $\overline{L}$ is a minimal set, then $L$ is compact or $\overline{L} - L$ is not closed. 
Hence $L$ is recurrent. 
If $\hat{L}$ is open, then $L$ is locally dense and so $L$ is recurrent. 
\end{proof}

The converse does not hold in general. 
In fact,  consider a codimension one foliation $\F'$ on $\T^3$ with one compact leaf and dense leaves. 
Then a foliation $\F := \{ L \times \{ x \} \mid L \in \F', x \in \S^1 \}$ on $\T^4$ is not $S_{1/2}$ but recurrent codimension two. 
However, the converse is true for codimension one $S_{1/4}$ foliations. 
In other words, 
codimension one recurrent $S_{1/4}$ foliations on a compact manifold are $S_{1/2}$. 
We will show it in the subsection \ref{subsec01}. 

From now on, let $\F$ be a codimension one foliation on a paracompact manifold $M$.

\subsection{Leaf properties}

We state the characterization of proper (resp. compact, minimal) foliations as follows. 

\begin{theorem}
Let $\F$ be a codimension one foliation on a compact connected manifold $M$.  
The following statement hold: 
\\
1.  
$\F$ is proper \iff $\F$ is $T_D$ \iff $\F$ is $T_{0}$. 
\\
2.  
$\F$ is compact \iff $\F$ is $T_{1}$ \iff $\F$ is $T_{2}$.  
\\
3.  
$\F$ is minimal \iff $\F$ is not $T_{0}$ but $S_1$.  
\end{theorem}

\begin{proof} 
Theorem 2.1 \cite{Y2} implies the equivalence 1. 
Since the leaf space of a continuous codimension one compact foliation 
of a compact manifold is either a closed interval or a circle 
(Corollary 5.3 \cite{Y3}), 
the equivalence 2 holds. 
Suppose that $\F$ is minimal. 
Obviously, $\F$ is not $T_{0}$ but $S_1$.  
Conversely, 
suppose that $\F$ is not $T_{0}$ but $S_1$.  
Then $\F$ consists of minimal sets 
and so $M = \min M = \max M$. 
Since exceptional leaves are not maximal, 
we obtain 
$M = \mathrm{Cl} \sqcup \mathrm{LD}$. 
Since $\F$ is not $T_{0}$, 
there is a locally dense leaf $L$. 
Since $\overline{L} = \hat{L}$, 
the connectivity of $M$ implies that 
$M = \hat{L}$. 
\end{proof} 

Recall that the superior structure $\mathop{\Downarrow_{\hat{\F}}}{L} = \cup \{ L' \in \F \mid L <_{\F} L' \}$ of a leaf $L$ is open,  and that the upset  $\mathop{\uparrow_{\F}}{L}= \cup \{ L' \in \F \mid L \geq_{\F} L' \}$ of an exceptional leaf $L$ is open connected (see  Lemma 2 and 3 \cite{S2}). 
These facts imply the following statements.

\begin{lemma}\label{lem:02a} 
Each locally dense leaf is maximal with respect to the order $<_{\F}$. 
\end{lemma}

\begin{proof}
For a locally dense leaf $L$ and for a leaf $L'$ with $L \subset \overline{L'}$, 
the local density implies that $L' \subset \hat{L}$.  
\end{proof}

\begin{lemma}\label{lem:002} 
If the height of $\tau_{\F}$ is finite, then the union $\mathrm{LD}$ is open 
and the closure $\overline{\mathrm{E}}$ is a finite union of closures of exceptional leaves. 
\end{lemma}

\begin{proof}
Since the codimension of $\F$ is one, the union $\mathrm{Cl}$ is closed. 
Suppose that the height of $\tau_{\F}$ is $k < \infty$. 
Theorem \cite{S} implies that 
$\mathrm{E}$ consists of finitely many local minimal sets. 
Lemma \ref{lem:02a} implies $\overline{\mathrm{E}} \cap \mathrm{LD} = \emptyset$.  
Then the complement $M -(\overline{\mathrm{E}} \cup \mathrm{Cl}) \subseteq \mathrm{LD} \sqcup \mathrm{P}$ is open. 
Let $U_1 := M -(\overline{\mathrm{E}} \cup \mathrm{Cl}) \subseteq \mathrm{LD} \sqcup \mathrm{P}$. 
Then $\mathrm{LD} \subseteq U_1$.  
By Theorem \cite{S}, 
the union $C_1$ of the leaves each of which is closed in the open subset $U_1$ 
is closed in $U_1$. 
By induction, 
define an open subset $U_{i+1} := U_i - C_i$. 
Then there is no leaf which is closed in the open subset $U_k$ 
and so $U_k \cap \mathrm{P} = \emptyset$. 
Thus  the complement $U_k = U_1 \setminus \mathrm{P} = \mathrm{LD}$ is open.  
\end{proof}

Note that there is a non-$S_{1/4}$ non-wandering codimension one foliation without exceptional leaves on a compact manifold of height three such that $\mathrm{P}$ is nether closed nor open. Indeed, 
we construct a foliated bundle over $\Sigma_3$ with one compact leaf, one proper leaf, and locally dense leaves, where $\Sigma_3$ is a closed orientable surface with genus three. 
Let $\S^1 = \R \sqcup \{ \infty \}$ be the one-point compactification of $\R$. 
Define three homeomorphisms $f, g$, and $h$ which are pairwise commutative as follows: 
Consider $f, g$ as commutative translation on $(0, 1)$ whose orbits are dense 
and extend $f, g: \R \to \R$ into homeomorphisms such that $f(x + 1) = f(x) + 1$ and $g(x + 1) = g(x) + 1$. 
Define $h: \R \to \R$ by $h(x) = x + 1$. 
Extend $f, g$, and $h$ into homeomorphisms by adding common fixed point $\infty$. 
Then the resulting foliated bundle whose total holonomy group is generated by $f, g$, and $h$ is desired, where the total holonomy group is the image of the monodromy $\pi_1 (\Sigma_3) \to \mathrm{Diff}(\S^1)$. 
Indeed, $M /\hat{\F}$ consists of three points $\hat{L}_{\infty}, \hat{L}_{cl}, \hat{L}_{\Z}$ such that $\hat{L}_{\infty}$ is a closed point, $\hat{L}_{cl}$ is an open point with $\overline{\hat{L}_{cl}} = M /\hat{\F}$, and $\hat{L}_{\Z}$ is neither a closed point nor an open point. 

%
On the other hand, the union $\mathrm{P}$ of foliations of height at most two is open.


\begin{corollary}\label{cor43}
The union $\mathrm{P}$ of non-compact proper leaves of an $S_{1/4}$ codimension one foliation on a compact manifold $M$ is open. 
\end{corollary}

\begin{proof}
By Lemma \ref{lem:01}, the height of $\tau_{\F}$ is at most one. 
By the decomposition of codimension one foliations, the union $\mathrm{Cl} \sqcup \mathrm{E} = M - (\mathrm{P} \sqcup \mathrm{LD})$ is closed. 
Assume that there is a locally dense leaf $L'$ with $\overline{L'} \cap \mathrm{P} \neq \emptyset$. 
Then the height of any locally dense leaf is at least two,  which contradicts to the hypothesis. 
Thus the boundary $\partial L$ of any locally dense leaf $L$ is contained in $\mathrm{Cl} \sqcup \mathrm{E}$, where $\partial A := \overline{A} - A$ is the boundary of a subset $A$.  
Assume that $\overline{\mathrm{LD}} \cap \mathrm{P} \neq \emptyset$. 
Fix a point $x \in \overline{\mathrm{LD}} \cap \mathrm{P}$. 
Since the closure of any locally dense leaf intersects to no points in $\mathrm{P}$, 
there are infinitely many locally dense leaves $L_i$ $( i \in \Z_{>0})$ such that 
$x \in \overline{\bigcup_{i \in \Z_{>0}} \partial L_i} \subseteq \mathrm{Cl} \sqcup \mathrm{E}$, which contradicts to $x \in \mathrm{P}$. 
\end{proof}

The non-existence of non-closed proper leaves implies 
the openness of locally dense leaves. 

\begin{lemma}\label{lem046}
Let $\F$ be a codimension one foliation on a compact manifold $M$ with $\mathrm{P} = \emptyset$. 
A leaf $L$ is locally dense \iff $\hat{L}$ is an open point in $M/\hat{\F}$. 
\end{lemma}

\begin{proof} 
Since the union $\mathrm{E} \sqcup \mathrm{Cl}$ is closed 
and $\mathrm{P} = \emptyset$,  
the complement  
$\mathrm{LD} = M - (\mathrm{E} \sqcup \mathrm{Cl})$ is open. 
Obviously, 
an open point in $M/\hat{\F}$ is locally dense. 
For a locally dense leaf $L$, 
since the leaf class $\hat{L}$ is closed in 
the open subset $\mathrm{LD}$, 
the intersection $\hat{L} \cap \mathrm{LD}$ is 
an open saturated \nbd of $L$ 
and so the leaf $L$ is an open point in $M/\hat{\F}$.  
\end{proof}

Since a foliation is recurrent if and only if $\mathrm{P} = \emptyset$, 
we state the following statements. 

\begin{proposition}\label{lem:12}
Let $\F$ be a codimension one recurrent foliation on a compact manifold $M$. 
A leaf $L$ is locally dense \iff $\hat{L}$ is an open point in $M/\hat{\F}$. 
\end{proposition}

\subsection{Dynamical-system-like properties of foliations}\label{subsec01}

Recall that the following statements are equivalent for a codimension one foliation $\F$ on a compact manifold $M$: 1) $\F$ is $S_{1}$; 2) $\F$ is $S_{2}$; 3) $\F$ is compact or minimal (i.e. $M = \mathrm{Cl}$ or $M = \mathrm{LD}$) \cite{Y3}. 
First, we have a characterization of $S_{1/4}$ foliations, to state dynamical-system-like properties of foliations.

\begin{lemma}\label{lem:31}
Let $\F$ be a codimension one foliation on a compact manifold $M$. 
Then the following are equivalent: 
\\
1. 
${\F}$ is $S_{1/4}$. 
\\
2. 
$\mathrm{P} \subseteq \max M$ 
and 
$\mathrm{E} \subseteq \min M$. 
\\
3. 
Each non-compact proper leaf is kerneled and the closure of each exceptional leaf is minimal. 

In any case, the unions $\mathrm{P}$ ane $\mathrm{LD}$ are open. 
\end{lemma}

\begin{proof} 
Obviously, 
the conditions 2 and 3 are equivalent. 
Since locally dense leaves are maximal 
and compact leaves are minimal, 
the condition 3 implies the condition 1. 
Conversely, 
suppose that the condition 1 holds. 
Since any non-compact proper leaf is not minimal, 
it is maximal and so kerneled. 
By Lemma 3 \cite{S}, 
each exceptional leaf is not maximal 
and so the closure of each exceptional leaf is minimal. 
\end{proof}

Note 
there is a codimension one foliation on a compact manifold 
whose leaf class space is not  $S_{1/3}$ but $S_{1/4}$. 
For instance, so is the Reeb foliation on the sphere $\S^3$.  
In fact, the complement $C := \S^3 - L$ of any planer leaf $L$  
is compact with respect to the quotient topology $\tau_{\F}$. 
Since the unique closed subset containing $C$ is the sphere $\S^3$ 
and the leaf $L$ is not a downset, there is no pair of a closed subset $F$ and 
a downset $D$ such that $C = F \setminus D$. 
We show the following equivalence for codimension one foliations. 

\begin{proposition}\label{lem47}
Let $\F$ be a codimension one foliation on a compact manifold $M$. 
Then the following are equivalent: 
\\
1. 
${\F}$ is $S_{1/2}$. 
\\
2. 
${\F}$ is $S_{1/3}$.  
\\
3. 
${\F}$ is $S_{1/4}$ 
and 
$\mathrm{P} = \emptyset$ 
\\
4. 
$\mathrm{E} \subseteq \min M$ and $\mathrm{P} = \emptyset$. 
\\
5. 
The closure of each leaf which is not locally dense is a minimal set.   
\end{proposition}

\begin{proof} 
Note hat $M = \mathrm{E} \sqcup \mathrm{Cl} \sqcup \mathrm{LD} \sqcup \mathrm{P}$. 
The decomposition implies that the condition 4 corresponds to the condition 5. 
By the definitions, an $S_{1/2}$ topology is $S_{1/3}$ (i.e. the condition 1 implies the condition 2).  
By Lemma \ref{lem:31}, the conditions 3 and 4 are equivalent. 
By Lemma \ref{lem046} and Lemma 2.3 \cite{Y}, the condition 3 
implies the condition 1. 
Suppose that $\tau_{\F}$ is $S_{1/3}$. 
Note that the unique $\F$-saturated \nbd of $\min M$ is the whole manifold $M$. 
Since $\min M$ is compact with respect to $\tau_{\F}$, 
any subset containing $\min M$ is compact with respect to the saturated topology $\tau_{\F}$. 
Assume that there is a non-compact proper leaf $L$. 
Then the complement $M - L$ is not closed but compact with respect to $\tau_{\F}$.
Since $M = \overline{M - L}$ and 
the proper leaf $L$ is not a downset, 
there is no pair of a closed $\F$-saturated subset $F$ and a downset $D$ such that 
$C = F \setminus D$, which contradicts to the $S_{1/3}$ separation axiom. 
Thus $\mathrm{P} = \emptyset$ and so the condition 3 holds.  
\end{proof}

There is a codimension one $S_{1/2}$ foliation on a compact manifold 
such that $\mathrm{LD}$ consists of 
infinitely many leaf classes. 
Indeed, we construct a foliated bundle over $\Sigma_2$ with infinitely many compact (resp. locally dense) leaves. 
Let $\S^1 = \R \sqcup \{ \infty \}$ be the one-point compactification of $\R$. 
Define two $C^{\infty}$ diffeomorphisms $f_1, f_2$ as follows: 
Consider diffeomorphisms $f_1, f_2: (0,1) \to (0,1)$ which are 
conjugate to commutative translations on $\R$ with relatively prime translation numbers. 
Then each orbit of the group generated by $f_1$ and $f_2$ is minimal. 
Extend $f_i: (0,1) \to (0,1)$ into a diffeomorphism such that  
$f_i(x + 1) = f_i(x) + 1$.  
Taking $\infty$ as a fixed point, 
extend $f_i: \S^1 \to \S^1$ into a diffeomorphism such that  
Then the resulting foliated bundle is desired. 
We show the following equivalence.

\begin{theorem}\label{cor:01}
Let $\F$ be a codimension one foliation on a compact manifold $M$. 
Then the following are equivalent: 
\\
1. $\F$ is $S_{1/2}$. 
\\
2. $\F$ is $S_{1/3}$. 
\\
3. $\F$ is recurrent and $S_{1/4}$. 
\\
4. 
$\F$ is recurrent and $\mathrm{E} \subseteq \min M$.  
\\
5. $M - \min M \subseteq \mathrm{LD}$.  
\end{theorem}

There are recurrent foliations with exceptional local minimal sets which are not minimal, which are not $S_{1/2}$. 
Indeed, fix $H_0$ be a finitely generated group action on a circle $\S^1 := \R/\Z$ which 
consists of a unique minimal set and dense orbits. 
Denote by $H_1$ the finitely generated group action on $\R$ generated by 
the lift of generators of $H$. 
Let $H_2$ the group action generated by $H_1$ and a translation $f : \R \to \R$ defined by $f(x) = x + 1$. 
Denote by $\S^1_{\infty} := \R \sqcup \{ \infty \}$ the one-point compactification of $\R$. 
Extend the group action $H_2$ on $\R$ into an action on $\S^1_{\infty}$ with the global fixed point $\infty$. 
Then the resulting finitely generated group-action on a circle $\S^1$ has 
a unique fixed point $\infty$,  one locally minimal set $\mathcal{E}$ which is exceptional, and locally dense orbits (i.e. $\mathrm{LD} = \R -  \mathcal{E}$). 
This implies that there is a non-$S_{1/2}$ codimension one foliated bundle on closed surface consists of one compact leaf, one exceptional local minimal sets which are not minimal, and locally dense leaves. 

%

\subsection{Several separation axioms for foliations}

Recall that a topological space $X$ is $S_D$ if 
the derived set $\overline{x} - \hat{x}$ of the class of any point $x \in X$ is closed.

\begin{lemma}\label{lem45}
A codimension one $S_{1/4}$ foliation on a compact manifold is $S_D$. 
\end{lemma}

\begin{proof} 
Let $L$ be a leaf of a codimension one $S_{1/4}$ foliation $\F$. 
If $L$ is closed or exceptional, 
then the inferior structure $\overline{L} - \hat{L}$ is empty and so closed. 
If $L$ is non-compact proper, 
then there is 
a saturated open \nbd $U$ of $L$ 
such that 
$\overline{L} \cap U = L$ 
and so 
$
\overline{L} - \hat{L}
= \overline{L} - L 
= \overline{L} \setminus U$ is closed.  
Thus we may assume that  $L$ is locally dense. 
 By Lemma \ref{lem:002} , the leaf class $\hat{L}$ is open and so 
the inferior structure $\overline{L} - \hat{L}$ is closed. 
\end{proof}


We characterize the $S_{YS}$  separation axiom using the $S_{1/2}$ separation axiom. 

\begin{lemma}\label{lem71} 
A codimension one foliation on a compact manifold is $S_{YS}$ 
if and only if it is $S_{1/2}$ such that $\overline{L} \cap \overline{L'} = \emptyset$ 
for any leaves $L, L \subset \mathrm{LD}$ with $\hat{L} \neq \hat{L'}$. 
\end{lemma}

\begin{proof}  
Let $\F$ be a codimension one foliation on a compact manifold $M$. 
Replacing $M$ with $M - \mathop{\mathrm{int}} \mathrm{Cl}$, 
we may assume that $\mathop{\mathrm{int}} \mathrm{Cl} = \emptyset$. 
Suppose that $\F$ is $S_{YS}$. 
Lemma \ref{lem61} implies that $\F$ is $S_{1/4}$ and that the leaf class space $M/\hat{\F}$ is a downward forest. 
Then $\mathrm{LD} \subseteq \max M$ 
and so $\overline{L} \cap \overline{L'} = \emptyset$ 
for any leaves $L, L \subset \mathrm{LD}$ with $\hat{L} \neq \hat{L'}$.
Conversely, 
suppose that $\F$ is $S_{1/2}$ such that $\overline{L} \cap \overline{L'} = \emptyset$ 
for any leaves $L, L \subset \mathrm{LD}$ with $\hat{L} \neq \hat{L'}$. 
By Proposition \ref{lem47}, 
there are no non-compact proper leaves and so $M = \mathrm{Cl} \sqcup \mathrm{LD} \sqcup \mathrm{E}$. 
Since the height of $M$ is one, 
we obtain that $\max M = M_1 \subseteq \mathrm{LD}$. 
By the hypothesis, the leaf class space $M/\hat{\F}$ is a downward forest of height at most one.
Lemma \ref{lem61} implies that $\F$ is $S_{YS}$. 
\end{proof}

Applying Lemma \ref{lem:01b} and Lemma \ref{lem:25}, 
we observe the following interpretations. 

\begin{lemma}
The following statement hold: 
\\
1.  
$\tau_{\F}$ is w-$C_0$ \iff $\F$ has no dense leaf. 
\\
2.  
$\tau_{\F}$ is $S_{Q}$ \iff $L \subseteq \overline{L'}$ or $L' \subseteq \overline{L}$ 
for any leaves $L, L'$ whose leaf closures intersect. 
\\
3.  
$\tau_{\F}$ is $C_N$ \iff each leaf closure contains exactly one minimal set. 
\end{lemma}

Note that 
the Reeb foliation on $\T^2$ is not $S_{1/3}$ but $T_{1/4}$  
such that the leaf space corresponds to the leaf class space and 
is a upward tree consisting of the bottom and a circle which consists of height one points 
(see Figure \ref{Reeb}). 
Indeed, the complement $\T^2 - L$ for a non-compact proper leaf $L$ 
is compact but not $\lambda$-closed with respect to $\tau_{\hat{\F}}$,  
because the unique closed subset containing $\T^2 - L$ is $\T^2$. 
\begin{figure}
\begin{center}
\includegraphics[scale=0.4]{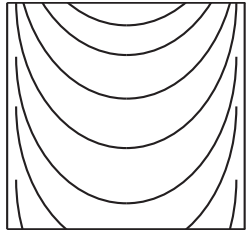}
\end{center}
\caption{A Reeb foliation of $\T^2$, which consists of one compact leaf and 
non-compact proper leaves}
\label{Reeb}
\end{figure}

The properties of the order $<_{\F}$ \cite{S} 
and the definition of ``Artinian'' 
imply the 
following two interpretations for foliations.

\begin{lemma}\label{lem:04} 
A foliation $\F$ is $S_D$ \iff $\overline{L} - \hat{L}$ is a union of finite leaf closures for any leaf $L \in \F$.  
\end{lemma}

\begin{proof}
By the properties of the inferior structure \cite{S}, 
the inferior structure $\overline{L} - \hat{L}$ is closed 
if and only if 
it is a finite union of leaf closures. 
\end{proof}

\begin{lemma}\label{lem:03} 
The class decomposition $\hat{\F}$ of a foliation $\F$ is Artinian \iff the height $\mathop{\mathrm{ht}}_{{\F}}(L)$ of each leaf $L \in \F$ is finite. 
\end{lemma}

These lemmas imply the following interpretation. 

\begin{proposition}\label{prop:c} 
The leaf class space of a codimension one foliation on a compact manifold is Artinian and $S_D$ \iff each leaf closure consists of finitely many local minimal sets.  
\end{proposition}

We show that recurrence implies quasi-Hausdorffness for codimension one foliations. 

\begin{proposition}\label{prop:01}
Each codimension one recurrent foliation on a compact manifold is q-$S_2$. 
\end{proposition}

\begin{proof} 
Let $\F$ be a codimension one recurrent foliation on a compact manifold $M$. 
Considering 
the transversally orientable foliation on the double covering of $M$ if necessary, 
we may assume that 
$\F$ is transversally orientable. 
If $L$ is exceptional, then Lemma 3 \cite{S} implies that there is an open saturated \nbd of $L$ where each leaf closure contains $L$, and so that $\mathop{\Uparrow}_{\F} L$ is an open saturated \nbd of $L$. 
By the recurrence, we have $M =  \mathrm{Cl}  \sqcup \mathrm{LD} \sqcup \mathrm{E}$. 
We may assume that there are leaves $L \neq L' \in \F$ such that there is no leaf $L'' \in \F$ with $L, L' \subseteq \overline{L''}$. 
Fix such leaves $L \neq L' \in \F$. 
Then $\mathop{\Uparrow}_{\F} L \cap \mathop{\Uparrow}_{\F} L' = \emptyset$. 
It suffices to show that $L$ and $L'$ can be separated by disjoint open saturated subsets. 
Indeed, suppose that $L$ and $L'$ are exceptional. 
Lemma 3 \cite{S} implies that $\mathop{\Uparrow}_{\F} L$ (resp. $\mathop{\Uparrow}_{\F} L'$) is an open saturated \nbd of $L$ (resp. $L'$). 
Thus we may assume that $L$ is not exceptional. 
Suppose that $L$ is locally dense. 
Then $\overline{L}$ is a closed saturated \nbd of $L$. 
By the hypothesis, $\overline{L} \cap L' = \emptyset$ and so $L$ and $L'$ are separated by disjoint open saturated subsets. 
Thus we may assume that $L$ and $L'$ are not locally dense. 
Then $L$ is compact. 
If the right (resp. left) holonomy of $L$ has a sequence of fixed points converging to a point in $L$, then there is a closed saturated right (resp. left) collar $V$ of $L$ which is a trivial product foliation on $[0, 1] \times L$ (resp. $[-1, 0] \times L$). 
Thus we may assume that the right (resp. left) holonomy of $L$ has no fixed point except points of $L$. 
Therefore there is a collar $C$ of $L$ which is homeomorphic to $[0, 1] \times L$ (resp. $[-1, 0] \times L$) such that each leaf closure of a leaf through a point in $C$ contains $L$. 
The saturation $\F(C)$ of $C$ is a collar of $L$ contained in $\mathop{\Uparrow}_{\F} L$ 
and so $\mathop{\Uparrow}_{\F} L$ contains an open saturated \nbd of $L$. 
If $L'$ is compact, then $\mathop{\Uparrow}_{\F} L'$ contains an open saturated \nbd of $L'$. 
If $L'$ is exceptional,  then Lemma 3 \cite{S} implies that $\mathop{\Uparrow}_{\F} L'$ is an open saturated \nbd of $L'$. 
This completes the proof. 
\end{proof}

Note that there are codimension one q-$S_2$ foliations which are not wandering (in particular, not recurrent) (e.g. Reeb foliations on some manifolds).   
In the group-action case, there is a $S_1$ group-action which is not q-$S_2$. 
In fact, a smooth vector field $v = (0, y)$ on $\T^2= (\R/\Z)^2$ 
generates such a group-action.  
Moreover, 
a smooth vector field $v = (1, \sin (2\pi z), 0)$ on $\T^3$ 
generates a codimension two $S_1$ foliation which is not q-$S_2$. 
On the other hand, 
the author would like to know whether 
there is a codimension one foliation on a compact manifold which is not q-$S_2$. 
In other words, 
is there a codimension one foliation on a compact manifold which is not q-$S_2$?

\subsection{Hyperbolic-like property for foliations}

Let $\F$ be a foliation on a compact manifold $M$. 

\begin{proposition}\label{lem:12}
Each codimension $k \geq 1$ proper foliation is weakly hyperbolic-like but not of Anosov type. 
\end{proposition}

\begin{proof}
Let $\F$ be a codimension $k$ proper foliation. 
Lemma \ref{lem:00a} implies that $\F$ is weakly hyperbolic-like. 
Since each leaf closure is not the whole manifold, $\F$ is not of Anosov type.  
\end{proof}

%
%

\end{document}